\documentclass[11pt]{amsart}
\usepackage{amsmath,amssymb,amsthm,amsfonts}
\usepackage[alphabetic]{amsrefs}
\usepackage{mathrsfs}
\usepackage[arrow,matrix,curve,cmtip,ps]{xy}
\usepackage{graphicx,tikz}
\usepackage[hypertexnames=true, citecolor = green, colorlinks = false, linkcolor = red, linktoc = none, pdffitwindow = false, urlbordercolor = white]{hyperref} 
\usepackage[shortlabels]{enumitem}
\usepackage[final]{pdfpages}
\usepackage{thmtools}
\usepackage{cleveref}
\usepackage{multirow}

\allowdisplaybreaks

\newcommand{\IC}[0]{\mathbb{C}}

 \newcommand{\IN}[0]{\mathbb{N}}

 \newcommand{\IT}[0]{\mathbb{T}}

\newcommand{\C}{\ensuremath{\mathbb{C}}}   

 \newcommand{\CF}[0]{\mathcal{F}}
 \newcommand{\CH}[0]{\mathcal{H}}
 
\newcommand{\CK}[0]{\mathcal{K}} \newcommand{\CL}[0]{\mathcal{L}}
 \newcommand{\CN}[0]{\mathcal{N}}
\newcommand{\CO}[0]{\mathcal{O}} 
 
 \newcommand{\CT}[0]{\mathcal{T}}
\newcommand{\CU}[0]{\mathcal{U}} 
 \newcommand{\CX}[0]{\mathcal{X}}

\newcommand{\supp}[0]{\operatorname{supp}}	
\newcommand{\Span}[0]{\operatorname{span}}
\newcommand{\id}[0]{\operatorname{id}}		

\newcommand*\into{\ensuremath{\lhook\joinrel\relbar\joinrel\rightarrow}}  

\newtheorem{theorem}{Theorem}[section]
\newtheorem{lemma}[theorem]{Lemma}
\newtheorem{proposition}[theorem]{Proposition}
\newtheorem{corollary}[theorem]{Corollary}
\theoremstyle{remark}
\newtheorem{remark}[theorem]{\bfseries Remark}

\newtheorem{definition}[theorem]{\bfseries Definition}
\newtheorem{example}[theorem]{\bfseries Example}

\newtheorem*{theorem*}{Theorem}
\newtheorem*{lemma*}{Lemma}
\newtheorem*{remark*}{Remark}

\numberwithin{equation}{section}

\newcommand{\Z}{\mathbb{Z}}
\newcommand{\N}{\mathbb{N}}

\makeatletter
\ifcsname phantomsection\endcsname
    \newcommand*{\qrr@gobblenexttocentry}[5]{}
\else
    \newcommand*{\qrr@gobblenexttocentry}[4]{}
\fi
\newcommand*{\addsubsection}{%
    \addtocontents{toc}{\protect\qrr@gobblenexttocentry}%
    \subsection}
\makeatother

\begin{document}
\title{Topological freeness for $*$-commuting covering maps}

\author{Nicolai Stammeier}
\address{Mathematisches Institut, Westf\"{a}lischen Wilhelms-Universit\"{a}t M\"{u}nster \\ \newline\hspace*{4mm}Einsteinstrasse 62 \\ 48149 M\"{u}nster \\ Germany}
\email{n.stammeier@wwu.de}

\thanks{Research supported by DFG through SFB $878$ and by ERC through AdG $267079$.}

\date{\today}



\begin{abstract}
A countable family of $*$-commuting surjective, non-injective local homeomorphisms of a compact Hausdorff space $X$ gives rise to an action $\theta$ of a countably generated, free abelian monoid $P$. For such a triple $(X,P,\theta)$, which we call an irreversible $*$-commutative dynamical system, we construct a universal $C^*$-algebra $\mathcal{O}[X,P,\theta]$. Within this setting we show that the following four conditions are equivalent: $(X,P,\theta)$ is topologically free, $C(X) \subset \mathcal{O}[X,P,\theta]$ has the ideal intersection property, the natural representation of $\mathcal{O}[X,P,\theta]$ on $\ell^2(X)$ is faithful, and $C(X)$ is a masa in $\mathcal{O}[X,P,\theta]$. As an application, we characterise simplicity of $\mathcal{O}[X,P,\theta]$ by minimality of $(X,P,\theta)$. We also show that $\mathcal{O}[X,P,\theta]$ is isomorphic to the Cuntz-Nica-Pimsner algebra of a product system of Hilbert bimodules naturally associated to $(X,P,\theta)$. Moreover, we find a close connection between $*$-commutativity and independence of group endomorphisms, a notion introduced by Cuntz and Vershik. This leads to the observation that, for commutative irreversible algebraic dynamical systems of finite type $(G,P,\theta)$, the dual model $(\hat{G},P,\hat{\theta})$ is an irreversible $*$-commutative dynamical system and $\mathcal{O}[\hat{G},P,\hat{\theta}]$ is canonically isomorphic to $\mathcal{O}[G,P,\theta]$. This allows us to conclude that minimality of $(G,P,\theta)$ is not only sufficient, but also necessary for simplicity of $\mathcal{O}[G,P,\theta]$ if $(G,P,\theta)$ is commutative and of finite type. 
\end{abstract}

\maketitle
\section*{Introduction}

\noindent Classically, the first object to study in the theory of topological dynamical systems is a single homeomorphism $\sigma$ of a compact Hausdorff space $X$, which induces an automorphism $\alpha$ of $C(X)$ via $\alpha(f) := f \circ \sigma$. The C*-algebra naturally associated to $(X,\sigma)$ is the crossed product $C(X)\rtimes_{\alpha}\Z$ generated by a copy of $C(X)$ and a unitary $u$ that implements $\alpha$ in the sense that $ufu^* = \alpha(f)$ holds for all $f \in C(X)$.

It is well known that the crossed product is simple if and only if the topological dynamical system is minimal in the sense that the only closed, $\sigma$-invariant subsets of $X$ are $\emptyset$ and $X$. Looking for a generalization of this result to the case of $\Z^{d}$-actions, minimality of $(X,\Z^{d},\sigma)$ alone turned out to be insufficient for simplicity of $C(X)\rtimes_{\alpha}\Z^{d}$, unless the action is free. This is automatic in the case of a single, minimal homeomorphism on an infinite space $X$ and means that $\sigma_{n}$ has no fixed points for all $n \neq 0$. Soon it turned out that simplicity of the transformation group C*-algebra does not detect the combination of minimality and freeness on the nose. Instead, one has to weaken freeness to topological freeness, where the set of fixed points of $\sigma_{n}$ is required to have empty interior for each $n \neq 0$, see \cites{KT,AS}. 

Interestingly, the proof of this correspondence exhibits the less prominent intermediate result that topological freeness of $(X,\Z^{d},\sigma)$ is characterized by the property that every non-zero ideal inside the transformation group C*-algebra intersects $C(X)$ non-trivially. This property is sometimes referred to as the \emph{ideal intersection property} (of $C(X)$ in $C(X)\rtimes_\alpha \Z^d$) and is actively studied for group crossed products, see for instance \cites{Sie,ST,dJST}. Additionally, building on \cite{Zel}, it has been observed that the ideal intersection property is equivalent to $C(X)$ being a maximal abelian subalgebra in the transformation group C*-algebra for amenable discrete groups, see \cite{KT}*{Theorem 4.1 and Remark 4.2}.

Doubtlessly, there is much more to say about the structure of group crossed products and we refer to \cites{BO} for an extensive and well-structured exposition. Instead, let us return to the case of a single transformation, which we now denote by $\theta$, and drop the reversibility of the system. One way of doing this in a moderate fashion is to demand that $\theta:X \longrightarrow X$ be a covering map, that is, a surjective local homeomorphism. Since we are interested in the irreversible case, we will assume that $\theta$ is not injective. Requiring that $\theta$ is a covering map has the convenient consequence that the induced map $\alpha:C(X) \longrightarrow C(X)$, given by $\alpha(f) = f \circ \theta$ is a unital, injective endomorphism. Moreover, $\theta$ is finite-to-one and the number of preimages $|\theta^{-1}(x)|$ of a singleton $x \in X$ is constant on the path-connected components of $X$. For simplicity, let us also assume throughout that this number is the same for all path-connected components of $X$. Such transformations are called regular in Definition~\ref{def:regular surj loc homeo}. Under these assumptions, there is a natural transfer operator $L$ for $\alpha$, see Example~\ref{ex:transfer operator for reg trafo}. In place of the group crossed product of $C(X)$ by $\Z$, it is reasonable to use the construction of a crossed product by an endomorphism $C(X)\rtimes_{\alpha,L}\N$ as introduced by Ruy Exel in \cite{Exe1}.

For this setup, Ruy Exel and Anatoly Vershik showed that $C(X)\rtimes_{\alpha,L}\N$ is simple if and only if $(X,\theta)$ is minimal, see \cite{EV}*{Theorem 11.3}. Their argument shows that topological freeness implies that $C(X)$ intersects every non-zero ideal in $C(X)\rtimes_{\alpha,L}\N$ non-trivially. But to the best of the author's knowledge, it was not until the work of Toke Meier Carlsen and Sergei Silvestrov that the equivalence of these two conditions was established in the irreversible setting described in the preceding paragraph, see \cite{CS}. In fact, their approach partially used results from \cite{EV} and incorporated two additional equivalent formulations. 

Now a countable family of $*$-commuting covering maps gives rise to what we will call an irreversible $*$-commutative dynamical system of finite type. Briefly speaking, we will show that the results and most ideas from \cite{CS} are extendible to this realm:

\begin{restatable*}{theorem}{TopFreeChar}\label{thm:top free char}
Suppose $(X,P,\theta)$ is an irreversible $*$-commutative dynamical system of finite type. Then the following statements are equivalent:
\begin{enumerate}
\item[$(1)$] The dynamical system $(X,P,\theta)$ is topologically free.
\item[$(2)$] Every non-zero ideal $I$ in $\CO[X,P,\theta]$ satisfies $I \cap C(X) \neq 0$.
\item[$(3)$] The representation $\varphi$ of $\CO[X,P,\theta]$ on $\ell^2(X)$ from Proposition~\ref{prop:elementary rep SIDoFT} is faithful.
\item[$(4)$] $C(X)$ is a masa in $\CO[X,P,\theta]$. 
\end{enumerate}
\end{restatable*} 

\noindent For this purpose, we first discuss some relevant features of $(X,P,\theta)$ and then construct the higher-dimensional analogue $\CO[X,P,\theta]$ of the Exel crossed product $C(X) \rtimes_{\alpha,L} \N$ for irreversible $*$-commutative dynamical systems of finite type, see Definition~\ref{def:O[X,P,theta]}. In order to employ the techniques from \cite{CS}, we also need a gauge-invariant uniqueness theorem for $\CO[X,P,\theta]$. This is achieved by realizing $\CO[X,P,\theta]$ as the Cuntz-Nica-Pimsner algebra of a product system of Hilbert bimodules over $P$ with coefficients in $C(X)$, see Theorem~\ref{thm:isom ad-hoc PS for SIDoFT}. This product system may be of independent interest for the study of $(X,P,\theta)$. As $P$ is a countably generated, free abelian monoid, we deduce that the gauge-invariant uniqueness theorem established in \cite{CLSV} holds for $\CO[X,P,\theta]$. Combining all these ingredients, we arrive at Theorem~\ref{thm:top free char}.

With this result at hands, it takes minor efforts to characterize simplicity of $\CO[X,P,\theta]$ by minimality of $(X,P,\theta)$, see Theorem~\ref{thm:simple iff minimal}. In view of the group case, this may seem a bit odd at first since topological freeness is not part of the characterization. But a modification of \cite{EV}*{Proposition 11.1} shows that minimal irreversible $*$-commutative dynamical systems of finite type are automatically topological free, see Proposition~\ref{prop:min -> top free}. 

It should be mentioned that one can derive a characterization of simplicity of $\CO[X,P,\theta]$ by considering the transformation groupoid associated to $(X,P,\theta)$. This has been accomplished in greater generality by Jonathan H. Brown, Lisa Orloff Clark, Cynthia Farthing and Aidan Sims, see \cite{BOFS}*{Theorem 5.1 and Corollary 7.8}. Moreover, one can deduce the equivalence of (1) and (2) out of \cite{BOFS}*{Proposition 5.5 and Proposition 7.5}. Nevertheless, the methods used here differ substantially from the ones in \cite{BOFS} and provide an account that is formulated entirely in the language of topological dynamical systems. Furthermore, the part involving conditions $(3)$ and $(4)$ is not covered by \cite{BOFS}.\vspace*{3mm}

\noindent However, this paper also admits a quite different perspective: The terminology of irreversible $*$-commutative dynamical systems of finite type $(X,P,\theta)$ hints at a close connection to irreversible algebraic dynamical systems $(G,P,\theta)$ as introduced in \cite{Sta1}. Recall that $(G,P,\theta)$ is given by a countable discrete group $G$ and a countable family of commuting, mutually independent, injective group endomorphisms of $G$, giving rise to an action $\theta$ of the monoid $P$ they generate. The two notions of dynamical systems are intended to capture different aspects of the motivating example $\times p,\times q:\IT \longrightarrow \IT$ for relatively prime $p,q \in \Z^\times$: While the algebraic dynamical system uses group theoretic properties, as for instance independence of group endomorphisms, the topological dynamical system considered in this paper relies on features like $*$-commutativity of covering maps. 

Nevertheless, we show that the notions of $*$-commutativity and (strong) independence are closely related, whenever both make sense, see Proposition~\ref{prop: eq star ind} for details. Clearly, this is the case if $G$ is commutative. Under this assumption, we prove that the dual model $(\hat{G},P,\hat{\theta})$ is an irreversible $*$-commutative dynamical systems of finite type if and only if $(G,P,\theta)$ is of finite type. 
We use this to find new examples for irreversible algebraic dynamical systems resembling the Ledrappier shift, see Example~\ref{ex:Ledrappier shift} and Example~\ref{ex:cellular automata}. 

In analogy to the construction of the C*-algebra $\CO[G,P,\theta]$ in \cite{Sta1}, we associate a C*-algebra $\CO[X,P,\theta]$ to each irreversible $*$-commutative dynamical system $(X,P,\theta)$ in Section~2. In fact, the construction is consistent in the sense that $\CO[\hat{G},P,\hat{\theta}] \cong \CO[G,P,\theta]$ holds for all commutative irreversible algebraic dynamical systems of finite type $(G,P,\theta)$, see Proposition~\ref{prop:consistent constr for CIADoFT}. As a corollary to this result and Theorem~\ref{thm:simple iff minimal}, we obtain that the criterion for simplicity of the C*-algebra $\CO[G,P,\theta]$ given in \cite{Sta1} is in fact necessary if $(G,P,\theta)$ is commutative and of finite type:  $\CO[G,P,\theta]$ is simple if and only if $(G,P,\theta)$ is minimal in the sense that $\bigcap_{p \in P}\theta_p(G) = \{1_G\}$, see Corollary~\ref{cor:char min-simple for CIADoFT}.\vspace*{3mm}

\noindent \textit{Acknowledgements.} This paper constitutes one part of the author's doctoral thesis completed at Westf\"alische Wilhelms-Universit\"at M\"unster under the supervision of Joachim Cuntz, whose support is highly appreciated.

\section{\texorpdfstring{Irreversible $*$-commutative dynamical systems}{Irreversible *-commutative dynamical systems}}
\noindent
This section is intended to familiarize the reader with the concept of $\ast$-commutativity so that we can present dynamical systems built from $*$-commuting surjective local homeomorphisms of a compact Hausdorff space that have a similar flavor as irreversible algebraic dynamical systems, see \cite{Sta1} for details. A close connection between strong independence and $\ast$-commutativity for commuting surjective group endomorphisms is established in Proposition~\ref{prop:*-comm vs ind}. In particular, this shows that the notion of $\ast$-commutativity coincides with independence for commutative irreversible algebraic dynamical systems of finite type. However, already the canonical shift action of $\N^2$ on $(\Z/2\Z)^{\N^2}$ provides an example where the two generators of the action do not $\ast$-commute but satisfy the independence condition.\vspace*{3mm}

\noindent The notion of $*$-commutativity was introduced by Victor Arzumanian and Jean Renault in $1996$ for a pair of maps $\theta_{1},\theta_{2}: X \longrightarrow X$ on an arbitrary set $X$, see \cite{AR}. For convenience, we will stick to the following equivalent formulation, see \cite{ER}*{Section 10}:

\begin{definition}
Suppose $X$ is a set and $\theta_{1},\theta_{2}: X \longrightarrow X$ are commuting maps. $\theta_{1}$ and $\theta_{2}$ are said to \textit{$\ast$-commute}, if, for every $x_{1},x_{2} \in X$ satisfying $\theta_{1}(x_{1}) = \theta_{2}(x_{2})$, there exists a unique $y \in X$ such that $x_{1} = \theta_{2}(y)$ and $x_{2} = \theta_{1}(y)$.
\end{definition}

\noindent The visualization of this property goes as follows: The maps $\theta_{1}$ and $\theta_{2}$ $\ast$-commute if and only if every diagram of the form
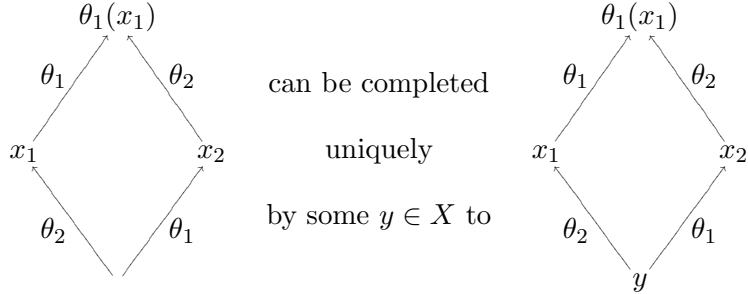
\[\scalebox{0.5}{\xymatrix{&\scalebox{2}{$\theta_{1}(x_{1})$}&& &&\scalebox{2}{$\theta_{1}(x_{1})$} \\ 
&&&\scalebox{2}{can be completed}\\ 
\scalebox{2}{$x_1$}\ar[uur]^{\scalebox{2}{$\theta_1$}} &&\scalebox{2}{$x_2$}\ar[uul]_{\scalebox{2}{$\theta_2$}}& \scalebox{2}{uniquely}& \scalebox{2}{$x_1$}\ar[uur]^{\scalebox{2}{$\theta_1$}} &&\scalebox{2}{$x_2$}\ar[uul]_{\scalebox{2}{$\theta_2$}}\\ 
&&&\scalebox{2}{by some $y \in X$ to}\\
&\text{\phantom{$y$}}\ar[uul]^{\scalebox{2}{$\theta_2$}}\ar[uur]_{\scalebox{2}{$\theta_1$}}&&&&\scalebox{2}{$y$}\ar[uul]^{\scalebox{2}{$\theta_2$}}\ar[uur]_{\scalebox{2}{$\theta_1$}}
}}\]

\begin{proposition}\label{prop: eq star ind}
Let $X$ be a set and $\theta_{1},\theta_{2}: X \longrightarrow X$ commuting maps. Then the following conditions are equivalent:
\begin{enumerate}[(i)]
	\item The maps $\theta_{1}$ and $\theta_{2}$ $\ast$-commute.
	\item For all $x \in X,~y_{1},y_{2} \in \theta_{1}^{-1}(x)$, $\theta_{2}(y_{1}) = \theta_{2}(y_{2})$ implies $y_{1} = y_{2}$.
	\item For all $x \in X$, the map $\theta_1:\theta_{2}^{-1}(x) \longrightarrow \theta_{2}^{-1}(\theta_{1}(x))$ is a bijection.
	\item[(iii')] For all $x \in X$, the map $\theta_2:\theta_{1}^{-1}(x) \longrightarrow \theta_{1}^{-1}(\theta_{2}(x))$ is a bijection.
\end{enumerate}
\end{proposition}
\begin{proof}
Observe that (ii) is basically a reformulation of (i), so their equivalence is straightforward. In order to see that (i) is equivalent to (iii), a diagram of the form
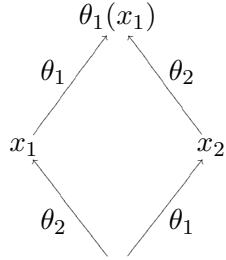
\[\scalebox{0.5}{\xymatrix{&\scalebox{2}{$\theta_{1}(x_{1})$} \\ 
\scalebox{2}{\phantom{l}}\\ 
\scalebox{2}{$x_1$}\ar[uur]^{\scalebox{2}{$\theta_1$}} &&\scalebox{2}{$x_2$}\ar[uul]_{\scalebox{2}{$\theta_2$}}&\scalebox{2}{\phantom{l}}\\ 
\scalebox{2}{\phantom{X}}\\
&\scalebox{2}{\phantom{$y$}}\ar[uul]^{\scalebox{2}{$\theta_2$}}\ar[uur]_{\scalebox{2}{$\theta_1$}}
}}\] 
clearly gives $x_2 \in \theta_2^{-1}(\theta_1(x_1))$. Thus, if we assume (iii), there is a unique $y \in \theta_2^{-1}(x_1)$ such that $\theta_1(y) = x_2$. In other words, the diagram can be completed uniquely. Conversely, if we assume (i), then, for every $x_2 \in \theta_2^{-1}(\theta_1(x_1))$, we get a unique $y_{x_2} \in \theta_2^{-1}(x_1)$ satisfying $\theta_1(y_{x_2}) = x_2$. Note that since $\theta_1$ and $\theta_2$ commute, we have $\theta_1(\theta_2^{-1}(x_1)) \subset \theta_2^{-1}(\theta_1(x_1))$. Hence, $x_2 \mapsto y_{x_2}$ is a bijection and, in fact, it is just the inverse map of $\theta_1$. The equivalence of (i) and (iii') follows by exchanging the role of $x_1,x_2$ and $\theta_1,\theta_2$.
\end{proof}

\noindent The following fact is probably well-known, but hard to track in the available literature, so we include a short proof based on Proposition~\ref{prop: eq star ind}.

\begin{lemma}\label{lem:star-comm under prod}
Let $X$ be a set and suppose $\theta_{1},\theta_2,\theta_3:X \longrightarrow X$ commute. $\theta_1$ $\ast$-commutes with $\theta_2\theta_3$ if and only if $\theta_1$ $\ast$-commutes with both $\theta_2$ and $\theta_3$.
\end{lemma}
\begin{proof}
Suppose $\theta_1$ $\ast$-commutes with $\theta_2\theta_3$. We will use the equivalent characterization of $\ast$-commutativity (ii) from Proposition~\ref{prop: eq star ind}. If we have $x \in X,y_1,y_2 \in \theta_1^{-1}(x)$ such that $\theta_2(y_1) = \theta_2(y_2)$, then $\theta_2\theta_3(y_1) = \theta_2\theta_3(y_2)$ forces $y_1 = y_2$. Thus $\theta_1$ and $\theta_2$ $\ast$-commute. For $\theta_1$ and $\theta_3$, we note that the situation is symmetric in $\theta_2$ and $\theta_3$. If $\theta_1$ $\ast$-commutes with both $\theta_2$ and $\theta_3$, then $\theta_1$ $\ast$-commutes with $\theta_2\theta_3$ by the equivalence of $\ast$-commutativity and condition (iii') in Proposition~\ref{prop: eq star ind}.
\end{proof}  

\noindent Given a compact Hausdorff space $X$, a first step away from reversibility is to consider local homeomorphisms instead of homeomorphisms. Let us recall that if $\theta:X \longrightarrow X$ is a local homeomorphism, then $|\theta^{-1}(x)|$ is finite for all $x \in X$. Indeed, the collection of all open subsets $U$ of $X$ on which $\theta$ is injective constitutes an open cover of $X$. By compactness of $X$, this can be reduced to a finite number which bounds $|\theta^{-1}(x)|$. 

We will be interested in surjective local homeomorphisms $\theta:X \longrightarrow X$ for which the cardinality of the preimage of a point is constant on $X$. Such transformations will be called regular. They also appear in \cite{CS} under the name \emph{covering map}. 

\begin{definition}\label{def:regular surj loc homeo}
Let $X$ be a compact Hausdorff space. A surjective local homeomorphism $\theta:X \longrightarrow X$ is said to be \textit{regular}, if $|\theta^{-1}(x)| = |\theta^{-1}(y)|$ holds for all $x,y \in X$.
\end{definition}

\noindent Via $f \mapsto f \circ \theta$, such a transformation yields an injective $\ast$-homomorphism $\alpha$ of $C(X)$ which has a left-inverse in the monoid formed by the positive linear maps $X \longrightarrow X$ with composition. This map can be defined abstractly on the C*-algebraic level:

\begin{definition}\label{def:transfer operator}
Given a C*-algebra $A$ and a $\ast$-endomorphism $\alpha$ of $A$, a positive linear map $L:A \longrightarrow A$ is called a \textit{transfer operator} for $\alpha$, if it satisfies $L(\alpha(a)b) = aL(b)$ for all $a,b \in A$. If $A$ is unital, $L$ is said to be \textit{normalized} provided that $L(1) = 1$.
\end{definition}

\begin{example}\label{ex:transfer operator for reg trafo}
If $X$ is a compact Hausdorff space and $\theta:X \longrightarrow X$ is a regular surjective local homeomorphism with $N_\theta := |\theta^{-1}(x)|$, where $x \in X$ is arbitrary, then 
\[\begin{array}{c} L(f)(x) := \frac{1}{N_\theta}\sum\limits_{y \in \theta^{-1}(x)}f(y) \end{array}\]
defines a transfer operator for the injective $\ast$-homomorphism $\alpha$ of $C(X)$ given by $f \mapsto f \circ \theta$. Indeed, $L$ is a positive linear map and, for $f,g \in C(X)$ and $x \in X$, we have
\[\begin{array}{c}L(\alpha(f)g)(x) = \frac{1}{N_\theta}\sum\limits_{y \in \theta^{-1}(x)}f(\theta(y))g(y) = (fL(g))(x). \end{array}\]
\end{example}

\begin{example}\label{ex:transfer operator disc ab fin type}
Let $G$ be a discrete abelian group and $\theta$ an injective group endomorphism of $G$ with $[G:\theta(G)] < \infty$. Then $\hat{\theta}$ is a local homeomorphism of $\hat{G}$. Moreover, $\hat{\theta}$ is surjective and every $k \in \hat{G}$ has $|\ker\hat{\theta}| = [G:\theta(G)]$ preimages under $\hat{\theta}$. Thus, $\hat{\theta}$ is regular. If $L$ is the transfer operator for $\hat{\theta}$ as in Example~\ref{ex:transfer operator for reg trafo} and $(u_g)_{g \in G}$ denote the standard generators of $C^*(G)$ , then 
\[L(u_g) = \chi_{\theta(G)}(g) u_{\theta^{-1}(g)}\]
holds for all $g \in G$. Indeed, if $g \in \theta(G)$, then $L(u_g) = L (\alpha(u_{\theta^{-1}(g)})) = u_{\theta^{-1}(g)}$, where $\alpha$ denotes the endomorphism $u_g \mapsto u_{\theta(g)}$. For the case $g \notin \theta(G)$, let $k \in \hat{G}$ and note that $\hat{\theta}^{-1}(k) = \ell_0 \ker\hat{\theta}$ holds for every $\ell_0 \in \hat{\theta}^{-1}(k)$. Hence, we get 
\[\begin{array}{c} L(u_g)(k) = \frac{1}{|\ker\hat{\theta}|} \sum\limits_{\ell \in \hat{\theta}^{-1}(k)} u_g(\ell) = \frac{1}{|\ker\hat{\theta}|} u_g(\ell_0)\sum\limits_{\ell \in \ker\hat{\theta}} u_g(\ell) = 0, \end{array}\]
since the sum over a finite, nontrivial subgroup of $\IT$ vanishes.   
\end{example}

\noindent It is well-known that for each normalized transfer operator $L$ for a unital $*$-endomorphism $\alpha$ of a unital C*-algebra $A$, the map $E:= \alpha \circ L$ is a conditional expectation from $A$ onto $\alpha(A)$. The next lemma is closely related to \cite{EV}*{Proposition 8.6}, so we omit its straightforward proof :

\begin{lemma}\label{lem:loc homeo right reconstruction}
Let $\theta:X \longrightarrow X$ be a regular surjective local homeomorphism of a compact Hausdorff space $X$ with $N_\theta := |\theta^{-1}(x)|$, where $x \in X$ is arbitrary. Denote by $L$ the natural transfer operator for the induced injective endomorphism $\alpha$ of $C(X)$. Then there exists a finite, open cover $\CU = (U_{i})_{1 \leq i \leq n}$ of $X$ such that the restriction of $\theta$ to each $U_{i}$ is injective. If $(v_{i})_{1 \leq i \leq n}$ is a partition of unity for $X$ subordinate to $\CU$, then $\nu_{i} := (N_\theta v_{i})^{\frac{1}{2}}$ satisfies 
\[\begin{array}{c} \sum\limits_{1 \leq i \leq n}{\nu_{i}\alpha \circ L(\nu_{i}f)} = f \text{ for all } f \in C(X). \end{array}\]
\end{lemma}

\noindent The equation proved in Lemma~\ref{lem:loc homeo right reconstruction} can be interpreted as a reconstruction formula. The conclusion of this result will be very important for the next two sections. Before we return to $\ast$-commuting maps, we add another small observation which is of independent interest. 

\begin{lemma}\label{lem:refine cover and pou}
Let $\theta_1,\theta_2:X \longrightarrow X$ be commuting continuous maps of a compact Hausdorff space $X$. Assume that there are two finite open covers $\CU_1 = (U_{1,i})_{i \in I_1}$ and $\CU_2 = (U_{2,i})_{i \in I_2}$ of $X$ such that $\theta_1|_{U_{1,i}}$ is injective for all $i \in I_1$ and $\theta_2|_{U_{2,i}}$ is injective for all $i \in I_2$. Then 
\[\CU_1 \vee_{\theta_1} \CU_2 := \left\{~ U_{1,i_1} \cap \theta_1^{-1}(U_{2,i_2}) ~|~ i_1 \in I_1,~i_2 \in I_2 ~\right\}\]
is a finite open cover of $X$ such that the restriction of $\theta_1\theta_2$ to every element of $\CU_1 \vee_{\theta_{1}} \CU_2$ is injective. Furthermore, suppose $(v_{1,i})_{i \in I_1}$ and $(v_{2,i})_{i \in I_2}$ are partitions of unity for $X$ subordinate to $\CU_1$ and $\CU_2$, respectively. If $\alpha_1$ denotes the endomorphism of $C(X)$ given by $f \mapsto f \circ \theta_1$, then $(v_{i_1,i_2})_{i_1 \in I_1,i_2 \in I_2}$, where $v_{i_1,i_2} := v_{1,i_1} \alpha_1(v_{2,i_2})$ defines a partition of unity subordinate to $\CU_1 \vee_{\theta_1} \CU_2$.
\end{lemma}
\begin{proof}
First of all, $\CU_1 \vee_{\theta_1} \CU_2$ consists of open sets by continuity of $\theta_1$ and it is clear that these sets cover $X$. If we let $U' := U_{1,i_1} \cap \theta_1^{-1}(U_{2,i_2})$, we get a commutative diagram:
\[\xymatrix{U' \ar[rr]^{\theta_1\theta_2}\ar[dr]_{\theta_1} &&X\\
&U_{2,i_2} \ar[ur]_{\theta_2}
}\] 
As $\theta_1$ is injective on $U_{1,i_1}$ and $\theta_2$ is injective on $U_{2,i_2}$, it follows that $\theta_1\theta_2$ is injective on $U_{1,i_1} \cap \theta_1^{-1}(U_{2,i_2})$ for all $i_1,i_2$. For the second part, we observe that  
\[\begin{array}{c} \sum\limits_{\substack{i_1 \in I_1 \\ i_2 \in I_2}}v_{i_1,i_2}(x) = \underbrace{\sum\limits_{i_1 \in I_1}v_{1,i_1}(x)}_{= 1} \underbrace{\sum\limits_{i_2 \in I_2}v_{2,i_2}(\theta_1(x))}_{= 1} = 1 \end{array}\]
holds for all $x \in X$ and 
\[\supp v_{i_1,i_2} = \supp v_{1,i_1} \cap \theta_1^{-1}(\supp v_{2,i_2}) \subset U_{1,i_1} \cap \theta_1^{-1}(U_{2,i_2}).\] 
\end{proof}



\noindent In particular, Lemma~\ref{lem:refine cover and pou} applies to commuting regular surjective local homeomorphisms by Lemma~\ref{lem:loc homeo right reconstruction}. The idea is to think of $\CU_1 \vee_{\theta_1} \CU_2$ as a common refinement of $\CU_1$ and $\CU_2$ with respect to $\theta_1$. But this construction is clearly not symmetric in $\theta_1$ and $\theta_2$.

The next proposition will be useful for the proof of Theorem~\ref{thm:isom ad-hoc PS for SIDoFT}.

\begin{proposition}\label{prop:star-com gives comm endo and transfer op}
Suppose $X$ is a compact Hausdorff space and $\theta_1,\theta_2: X \longrightarrow X$ are regular surjective local homeomorphisms. Let $\alpha_i$ denote the endomorphism of $C(X)$ induced by $\theta_i$ and be $L_1$ the natural transfer operator for $\alpha_1$ as constructed in Example~\ref{ex:transfer operator for reg trafo}. Then $\theta_1$ and $\theta_2$ $*$-commute if and only if $L_1$ and $\alpha_2$ commute.
\end{proposition}
\begin{proof}
Assume that $\theta_1$ and $\theta_2$ $*$-commute. Using (iii') from Proposition~\ref{prop: eq star ind}, this is a straightforward computation. For $f \in C(X)$ and $x \in X$, we get
\[\begin{array}{c}
L_1(\alpha_2(f))(x) = \frac{1}{N_1} \sum\limits_{z \in \theta_2(\theta_1^{-1}(x))}f(z) = \frac{1}{N_1} \sum\limits_{z \in \theta_1^{-1}(\theta_2(x))}f(z) = \alpha_2(L_1(f))(x).
\end{array}\] 
If $\theta_1$ and $\theta_2$ do not $*$-commute, there is $x \in X$ such that $\theta_2(\theta_1^{-1}(x))$ is a proper subset of $\theta_1^{-1}(\theta_2(x))$ because $\theta_1$ is regular. Thus we conclude $\alpha_2(L_1(1))(x) \neq L_1(\alpha_2(1))(x)$, so $L_1$ and $\alpha_2$ do not commute.
\end{proof}

\noindent Next, we will define the analogue of an irreversible algebraic dynamical system of finite type based on $\ast$-commuting regular transformations of a compact Hausdorff space $X$, compare \cite{Sta1}. As $X$ is compact, we cannot get anything beyond the finite type case here. We note that more general dynamical systems of this type have been considered in \cite{FPW}, where $X$ is allowed to be locally compact. In their approach, regularity is relaxed to the requirement that there is a uniform finite bound on the number of preimages of a single point, see \cite{FPW}*{Definition 3.2}.

\begin{restatable}{definition}{SID}\label{def:SID} 
An \textit{irreversible $*$-commutative dynamical system of finite type} is a triple $(X,P,\theta)$ consisting of 
\begin{enumerate}[(A)]
\item a compact Hausdorff space $X$,
\item a countably generated free abelian monoid $P$ with unit $1_P$ and 
\item an action $P \stackrel{\theta}{\curvearrowright} X$ by regular surjective local homeomorphisms with the following property: $\theta_p$ and $\theta_q$ $\ast$-commute if and only if $p$ and $q$ are relatively prime in $P$. 
\end{enumerate}
\end{restatable}

\noindent Before considering examples, let us recall the notions of (strong) independence for surjective group endomorphisms introduced in \cite{Sta1}.

\begin{definition}\label{def:ind surj}
Two commuting, surjective group endomorphisms $\theta_{1}$ and $\theta_2$ of a group $K$ are said to be \textit{strongly independent}, if they satisfy $\ker\theta_{1} \cap \ker\theta_{2} = \{1_K\}$. $\theta_{1}$ and $\theta_2$ are called \textit{independent}, if $\ker\theta_1 \cdot \ker\theta_2 = \ker\theta_1\theta_2$ holds. 
\end{definition}

\noindent We will see that independence is directly connected to $*$-commutativity in the case of surjective group endomorphisms. In fact, independence turns out to be weaker in principle, but the two conditions are equivalent if the kernel of one of the surjective group endomorphisms is a co-Hopfian group. Co-Hopfian groups have first been studied under the name S-groups in \cite{Bae} and we refer to \cites{GG,ER2} as well as \cite{dlH}*{Section 22 of Chapter III} for more information on the subject.

\begin{definition}\label{def:co-Hopfian gp}
A group $K$ is said to be co-Hopfian if every injective group endomorphism $\theta:K \into K$ is already an automorphism of $K$.
\end{definition}


\begin{proposition}\label{prop:*-comm vs ind}
Suppose $K$ is a group and $\theta_1,\theta_2$ are commuting surjective endomorphisms of $K$. If $\theta_1$ and $\theta_2$ $\ast$-commute, then $\theta_1$ and $\theta_2$ are strongly independent. If $\theta_2:\ker\theta_1 \longrightarrow \ker\theta_1$ or $\theta_1:\ker\theta_2 \longrightarrow \ker\theta_2$ is surjective, then the converse holds as well. In particular, this is the case if $\ker\theta_1$ or $\ker\theta_2$ is co-Hopfian.
\end{proposition}
\begin{proof}
Note that we have $\theta_i^{-1}(k) = k' \ker\theta_i$ for all $k \in K$ where $k' \in \theta_i^{-1}(k)$ is chosen arbitrarily. According to Proposition~\ref{prop: eq star ind}, $\theta_1$ and $\theta_2$ $\ast$-commute precisely if 
$\theta_1:\theta_2^{-1}(k) \longrightarrow \theta_1(\theta_2^{-1}(k))$ is bijective for all $k \in K$. Since $\theta_1$ and $\theta_2$ are group endomorphisms, this is equivalent to the requirement that $\theta_1$ is an automorphism of the subgroup $\ker\theta_2$. Indeed, this is clearly necessary and if it is true, then $\theta_1:\theta_2^{-1}(k) \longrightarrow \theta_1(\theta_2^{-1}(k))$ is a bijection because $\theta_2^{-1}(\theta_1(k)) = \theta_1(k')\ker\theta_2$ and $\theta_1(\theta_2^{-1}(k)) = \theta_1(k')\theta_1(\ker\theta_2)$. In particular, we have $\ker\theta_1\cap\ker\theta_2 = \{1_K\}$, so $\theta_1$ and $\theta_2$ are strongly independent in the sense of Definition~\ref{def:ind surj}. Moreover, we see that strong independence corresponds to injectivity of $\theta_1$ and $\theta_2$ on $\ker\theta_2$ and $\ker\theta_1$, respectively. Hence, if one of these maps is surjective, we get $\ast$-commutativity of $\theta_1$ and $\theta_2$. By definition, this is for granted if one knows that one of the kernels is a co-Hopfian group.
\end{proof}

\noindent Recall from \cite{Sta1} that an irreversible algebraic dynamical system $(G,P,\theta)$ consists of 
\begin{enumerate}[(A)]
\item a countably infinite discrete group $G$ with unit $1_G$,
\item a countably generated free abelian monoid $P$ with unit $1_P$, and 
\item an action $P \stackrel{\theta}{\curvearrowright} G$ by injective group endomorphisms with the property that $\theta_p$ and $\theta_q$ are independent if and only if $p$ and $q$ are relatively prime in $P$. 
\end{enumerate}
Moreover, $(G,P,\theta)$ is said to be commutative if $G$ is commutative and it is said to be of finite type, if the index $[G:\theta_p(G)]$ is finite for all $p \in P$.

\begin{corollary}\label{cor:CIAD vs SIDoFT}
Let $G$ be a discrete abelian group, $P$ a monoid and $P \stackrel{\theta}{\curvearrowright}G$ an action by group endomorphisms. $(G,P,\theta)$ is a commutative irreversible algebraic dynamical system of finite type if and only if $(\hat{G},P,\hat{\theta})$ is an irreversible $*$-commutative dynamical system of finite type.
\end{corollary}
\begin{proof}
This follows readily from Proposition~\ref{prop:*-comm vs ind}.
\end{proof}

\noindent This last result provides examples for irreversible $*$-commutative dynamical systems of finite type coming from commutative irreversible algebraic dynamical systems of finite type, see \cite{Sta1} for details.\label{rem:ind less restrictive than *-comm}

There are also interesting examples of dynamical systems built from $\ast$-commuting transformations in symbolic dynamics, see for instance \cite{ER}*{Sections 10--14} and \cites{Wil2,MW}. On the other hand, $\ast$-commutativity is also considered to be a severe restriction. While $\ast$-commutativity implies strong independence in the case of surjective group endomorphisms, there are examples for commutative irreversible algebraic dynamical systems that do not satisfy the strong independence condition, see \cite{Sta1}.

For the remainder of this section, we would like to direct the reader's attention to another intriguing class of examples for irreversible algebraic dynamical systems, namely to dynamical systems arising from cellular automata. This part builds on \cite{ER}*{Section 14} and can be considered as a natural extension of the observations presented there. In the following, let $X = (\Z\big/2\Z)^\N$ and $\sigma$ denote the unilateral shift, i.e. $\sigma(x)_k = x_{k+1}$ for all $k \in \N$ and $x \in X$. Moreover, let $X_n = (\Z\big/2\Z)^n$ for $n \in \N$ and suppose we are given $D \subset X_n$. Then we can define a transformation $\theta_D$ of $X$ by the sliding window method
\[(\theta_D(x))_k = \chi_{D}(x_k,x_{k+1},\dots,x_{k+n-1}).\] 
In other words, the entry at place $k$ becomes $1$ if the word of length $n$ starting at place $k$ belongs to the so-called dictionary $D$. It is interesting to analyze the extent to which properties of $\theta_D$ can be expressed in terms of its dictionary. One outcome of such considerations are the following two definitions: 

\begin{definition}\label{def:adm dic}For $n \in \N$, a subset $D \subset X_n$ is called a \textit{dictionary}. $D$ is called \textit{progressive}, if for any $x \in X_{n-1}$, there is a unique $x_n \in X_1$ such that $(x_1,\dots,x_n) \in D$. $D$ is called \textit{admissible}, if it is progressive and has the property that, for $x,y,z \in X_n$, $x+y = z\in D$ implies that either $x \in D$ or $y \in D$ holds.
\end{definition}

\noindent Let us observe that $X_n{\setminus}D$ is a group of order $2^{n-1}$ for every admissible dictionary $D$. It is clear that $\theta_D$ is continuous on $X$ and commutes with $\sigma$ for every dictionary $D$. Morton L.~Curtis, Gustav A.~Hedlund and Roger Lyndon have shown in \cite{Hed} that any continuous self-map of $X$ which commutes with the shift $\sigma$ corresponds to a cellular automaton (Even though the article is authored by Hedlund only, he credits Curtis and Lyndon as co-discoverers in the introduction.). Thus $(X,\theta_D)$ can be identified as a cellular automaton. It is shown in \cite{ER}*{Theorem 14.3} that for progressive $D$, the transformation $\theta_D$ is a surjective local homeomorphism of $X$. This allows us to deduce:

\begin{proposition}\label{prop:adm dic}
If $D \subset X_n$ is an admissible dictionary, then $\theta_D$ is a continuous surjective group endomorphism of $X$ that commutes with $\sigma$. $\ker\theta_D$ is isomorphic to the group $X_n{\setminus}D$ and thus consists of $2^{n-1}$ elements.
\end{proposition}

\begin{remark}\label{rem:adm dic fin kernel}
In view of Proposition~\ref{prop: eq star ind}, we are now in position to provide new examples for commutative irreversible algebraic dynamical systems of finite type in terms of their dual pictures. We note that it is easier to check strong independence of $\sigma$ and $\theta_D$ than examining $\ast$-commutativity of these for an admissible $D$. Indeed, $\ker\sigma$ is easily determined and Proposition~\ref{prop:adm dic} provides us with an explicit description of $\ker\theta_D$.
\end{remark}
 
\noindent A guiding example is the Ledrappier shift, see \cite{ER}*{Section 11}:

\begin{example}\label{ex:Ledrappier shift}
Let $Y$ be the subshift of $(\Z\big{/} 2\Z)^{\N^{2}}$ given by all sequences $y = (y_{n})_{n \in \N^{2}}$ s.t. $y_{n}+y_{n+e_{1}}+y_{n + e_{2}} = 0 \in \Z\big/2\Z$ for all $n \in \N^{2}$. $\N^{2} \stackrel{\theta}{\curvearrowright} Y$ is given by the coordinate shifts $\theta_{e_{i}}(y_{n})_{n} = (y_{n+e_{i}})_{n},~i = 1,2$. The four basic blocks in $Y$ are:
\begin{center}
\begin{tikzpicture}
\matrix [nodes=draw,column sep=0.1mm, row sep=0.1mm]
{
\node {0}; &&[15mm] \node {1}; &&[15mm] \node {0}; &&[15mm] \node {1}; \\
\node {0}; & \node {0}; & \node {0}; & \node {1}; &  \node{1}; & \node {1}; &  \node {1}; & \node {0}; \\
};
\end{tikzpicture}
\end{center}
Observe that, for any given $y \in Y$ and every path $(n_{m})_{m \in \N}$ with $n_{m+1} \in \{n_{m}+e_{1},n_{m}+e_{2}\}$, the sequence $(y_{n_{m}})_{m \in \N}$ determines $y$ completely. Conversely, one can show inductively, that for every path $(n_{m})_{m \in \N}$ and sequence $(y_{m})_{m \in \N}$ with $y_{m} \in \Z \big/ 2\Z$, there is an $y \in Y$ with $y_{n_{m}} = y_{m}$ for all $m$. One consequence of this is that there is a homeomorphism $Y \longrightarrow X = (\Z\big/2\Z)^\N$ given by restricting to the base row, i.e. $(y_{m,n})_{m,n \in \N} \mapsto (y_{n,0})_{n \in \N}$. Under this homeomorphism to the Bernoulli space, $\theta_{e_1}$ corresponds to the shift $\sigma$ on $X$ and $\theta_{e_2}$ corresponds to $x \mapsto x+\sigma(x) = (x_n+x_{n+1})_{n \in \N}$ for $x \in X$. In view of the example from cellular automata, it is quite intriguing to notice that the Ledrappier shift fits into the picture quite nicely: The conjugate map to the vertical shift is nothing but $\theta_D$ for the admissible dictionary $D = \{(0,1),(1,0)\}$. In fact, $(X,\theta_D)$ is the most basic non-trivial example of a cellular automaton coming from an admissible dictionary. By Proposition~\ref{prop:adm dic}, $\theta_D$ $\ast$-commutes with the shift, so $\theta_{e_1}$ and $\theta_{e_2}$ $\ast$-commute. Hence the Ledrappier shift gives rise to a commutative irreversible algebraic dynamical system of finite type.
\end{example}

\begin{remark}\label{rem:Ledrappier type ex from adm dic} 
We have seen that the Ledrappier shift can be obtained from an admissible dictionary. In fact, there is only one admissible dictionary $D$ for words of length $2$ such that the induced transformation $\theta_D$ $\ast$-commutes with shift $\sigma$. So the Ledrappier shift constitutes a minimal non-trivial example of a commutative irreversible algebraic dynamical system of finite type arising from a cellular automaton.
\end{remark}

\noindent Reversing the perspective, the Ledrappier shift is formed out of the cellular automaton $(X,\theta_D)$ by stacking the orbit. This is to say that for $x \in (\Z\big/2\Z)^\N$ the $k$-th row of the corresponding element in $(\Z\big/2\Z)^{\N^2}$ is given by $\theta_D^k(x)$. Building on this observation, we may always construct a subshift of $(\Z\big/2\Z)^{\N^2}$ out of a progressive dictionary. This may turn out to be a source of potentially interesting subshifts of $(\Z\big/2\Z)^{\N^2}$. Let us now look at what happens for dictionaries using longer words:

\begin{example}\label{ex:cellular automata}
Let $D_1,D_2 \subset X_3$ be the dictionaries
\[\begin{array}{rcl}
D_1 &=& \{(0,0,1),(1,0,0),(0,1,1),(1,1,0)\}\\
&\text{and}\\
D_2 &=& \{(0,0,1),(1,0,0),(0,1,0),(1,1,1)\}.
\end{array}\]  
Then $D_1$ and $D_2$ are admissible dictionaries. Hence, $\theta_{D_1}$ and $\theta_{D_2}$ are surjective group endomorphisms of $X = (\Z\big/2\Z)^\N$ that commute with the shift $\sigma$ and 
\[\begin{array}{rcl}
\ker\theta_{D_1} &=& \{0,1,(\overline{0,1},\dots),(\overline{1,0},\dots)\}, \\
\ker\theta_{D_2} &=& \{0,(\overline{1,0,1},\dots),(\overline{0,1,1},\dots),(\overline{1,1,0},\dots)\},
\end{array}\]
where we write $(\overline{a,b,c},\dots)$ for the periodic word $(a,b,c,a,b,c,\dots)$. Apparently, we have $\ker\sigma = \{0,(1,\overline{0},\dots)\}$, so $\sigma$ and $\theta_{D_i}$ are strongly independent for $i = 1,2$. By Proposition~\ref{prop: eq star ind}, they also $\ast$-commute. Hence, each $D_i$ gives rise to a commutative irreversible algebraic dynamical system of finite type $(G,P,\theta)$ with $G = \hat{X}$ and $P = |\sigma, \theta_{D_i}\rangle \cong \N^2$ acting by their dual endomorphisms. Noting that $\ker\theta_{D_1} \cap \ker\theta_{D_2}$ is trivial, we also get a commutative irreversible algebraic dynamical system of finite type for $P= |\sigma, \theta_{D_1}, \theta_{D_2}\rangle \cong \N^3$. 
\end{example}

\begin{remark}\label{rem:adm dics for length 3} 
In fact, $D_1$ and $D_2$ are the only admissible dictionaries for words of length $3$ for which the induced transformation $*$-commutes with $\sigma$. Indeed, every such admissible dictionary $D$ needs to contain $(0,0,1)$ and $(1,0,0)$. If $(0,0,0) \in D$, then $D$ cannot induce a group homomorphism. Likewise, if we had $(1,0,0) \notin D$, then $\ker\theta_D$ would contain $\ker\sigma$. In particular, their intersection would be non-trivial. Now, if $(0,1,1) \in D$, then this forces $(1,1,0) \in D$ since $(0,1,1)+(1,1,1) = (1,0,0) \in D$. Similarly, $(0,1,0) \in D$ forces $(1,1,1) \in D$ since $(0,1,0)+(1,1,0) = (1,0,0) \in D$. One can check that there are precisely two additional admissible dictionaries $D_3,D_4 \subset X_3$ given by 
\[\begin{array}{rcl}
D_3 &=& \{(0,0,1),(1,0,1),(0,1,0),(1,1,0)\}\\
&\text{and}\\
D_4 &=& \{(0,0,1),(1,0,1),(0,1,1),(1,1,1)\}.
\end{array}\]  
Thus, there are four admissible dictionaries for word length $3$, two of which induce surjective group endomorphisms of $X$ that $\ast$-commute with the shift $\sigma$. The corresponding group endomorphisms of $X$ are
\[\begin{array}{rclcrcl}
\theta_1(x) &=& x + \sigma^2(x)&\hspace*{4mm}\multirow{2}{*}{and}\hspace*{4mm}& \theta_3(x) &=& \sigma(x) + \sigma^2(x)\\
\theta_2(x) &=& x + \sigma(x) + \sigma^2(x)&& \theta_4(x) &=& \sigma^2(x).
\end{array}\]
This simple description raises the question whether it might be possible to characterize admissibility of a dictionary $D \subset X_n$  for general $n \geq 2$ and $*$-commutativity of $\theta_{D}$ with $\sigma$ in a more accessible way.
\end{remark}

\begin{remark}\label{rem:comparison to ER-ex}
In \cite{ER}*{Example 14.4}, Ruy Exel and Jean Renault provide an example of a progressive dictionary which does not induce a transformation that $\ast$-commutes with the shift, namely 
\[D = \{(0,0,0),(1,0,0),(0,1,0),(1,1,1)\}.\] 
This is stated implicitly in \cite{ER}*{Corollary 14.5} and follows from \cite{ER}*{Theorem 10.4 and Proposition 14.1}. However, this dictionary does not give a group homomorphism of $X$ because it contains the neutral element of $X$ and hence $\theta_D(0) \neq 0$. The dictionary $D_2$ from Example~\ref{ex:cellular automata} is a slight variation of \cite{ER}*{Example 14.4} designed to produce a group homomorphism.
\end{remark}

\section{C*-algebras with reconstruction formulas}
\noindent
This section is devoted to the construction of universal C*-algebras for irreversible $*$-commutative dynamical systems of finite type $(X,P,\theta)$. We show that this construction is consistent with the natural realization of $(X,P,\theta)$ as operators on $\ell^2(X)$, see Proposition~\ref{prop:elementary rep SIDoFT}. Moreover, we show that, for commutative irreversible algebraic dynamical systems of finite type $(G,P,\theta)$, there is a natural isomorphism between $\CO[G,P,\theta]$ and $\CO[\hat{G},P,\hat{\theta}]$, see Proposition~\ref{prop:consistent constr for CIADoFT}. In addition, we establish a few elementary properties for $\CO[X,P,\theta]$ and its core subalgebra $\CF$. A fair amount of the results from this section is relevant for Section~4.  

Throughout this section, $(X,P,\theta)$ denotes an irreversible $*$-commutative dynamical system, unless specified otherwise. Recall that, for $p \in P$, the endomorphism $\alpha_p$ of $C(X)$ and its transfer operator $L_p$ are given by
\[\begin{array}{c} \alpha_p(f)(x) = f(\theta_p(x)) \text{ and } L_p(f)(x) = \frac{1}{N_p}\sum\limits_{y \in \theta_p^{-1}(x)}\hspace*{-2mm}f(y) \text{ for } x \in X,f \in C(X), \end{array}\]
where $N_p = |\theta_p^{-1}(x)|$. Moreover, we let $E_p := \alpha_p \circ L_p:C(X) \longrightarrow \alpha_p(C(X))$ denote the corresponding conditional expectation.
 
\begin{restatable}{definition}{SIDoFTalg}\label{def:O[X,P,theta]}
$\CO[X,P,\theta]$ is the universal C*-algebra generated by $C(X)$ and a representation of the monoid $P$ by isometries $(s_{p})_{p \in P}$ subject to the relations:
\[\begin{array}{rrrcll}
(I)&&s_pf &\hspace*{-1mm}=\hspace*{-1mm}& \alpha_p(f)s_p&\text{ for all } f \in C(X),p\in P.\\
(II)&&s_p^*fs_p &\hspace*{-1mm}=\hspace*{-1mm}& L_p(f)&\text{ for all } f \in C(X),p\in P.\\
(III)&&s_p^*s_q &\hspace*{-1mm}=\hspace*{-1mm}& s_qs_p^*&\text{ if $p$ and $q$ are relatively prime.}\\
(IV)&\multicolumn{5}{l}{\text{If $p \in P$ and $f_{1,1},\dots,f_{n,1},f_{1,2},\dots,f_{n,2} \in C(X)$, satisfy}}\vspace*{2mm}\\ 
&\sum\limits_{1 \leq i \leq n}&{f_{i,1}E_p(\bar{f}_{i,2}f)} &\hspace*{-1mm}=\hspace*{-1mm}& f &\text{ for all } f \in C(X), \text{then}\vspace*{2mm}\\
 &\sum\limits_{1 \leq i \leq n}&{f_{i,1}s_ps_p^*\bar{f}_{i,2}} &\hspace*{-1mm}=\hspace*{-1mm}& 1.
\end{array}\]
\end{restatable}

\noindent The next lemma explains the motivation behind relation (IV).

\begin{lemma}\label{lem:CNP equiv}
For every $p \in P$, the validity of relation (IV) from Definition~\ref{def:O[X,P,theta]} is independent of the choice of the family $(f_{i,j})_{1 \leq i \leq m, j = 1,2}$ satisfying the reconstruction formula. In particular, if $\CU = (U_{i})_{1 \leq i \leq n}$ is a finite open cover of $X$ such that the restriction of $\theta_{p}$ to each $U_{i}$ is injective and $(v_{i})_{1 \leq i \leq n}$ is a partition of unity for $X$ subordinate to $\CU$, then
\[\begin{array}{c}\sum\limits_{1 \leq i \leq n}{\nu_{i}s_{p}s_{p}^{*}\nu_{i}} = 1 \end{array}\]
holds for $\nu_{i} = (N_{p}v_{i})^{\frac{1}{2}}$.
\end{lemma}
\begin{proof}
For the first part, let $(f_{i,j})_{1 \leq i \leq m, j = 1,2}$ and $(g_{k,\ell})_{1 \leq k \leq n, \ell = 1,2}$ be two families in $C(X)$ that both satisfy the reconstruction formula for all $f \in C(X)$. Now if relation (IV) from Definition~\ref{def:O[X,P,theta]} holds for $(f_{i,j})_{1 \leq i \leq m, j = 1,2}$, then 
\[\begin{array}{rcl}
\sum\limits_{1 \leq k \leq n}{g_{k,1}s_ps_p^*\bar{g}_{k,2}} &=& \sum\limits_{1 \leq k \leq n}{g_{k,1}s_ps_p^*\bar{g}_{k,2}}\hspace{2mm} \sum\limits_{1 \leq i \leq m}{f_{i,1}s_ps_p^*\bar{f}_{i,2}} \vspace*{2mm}\\
&=& \sum\limits_{1 \leq i \leq m}\hspace{2mm}\sum\limits_{1 \leq k \leq n}{g_{k,1} E_p(\bar{g}_{k,2}f_{i,1})s_ps_p^*\bar{f}_{i,2}} \vspace*{2mm}\\
&=& \sum\limits_{1 \leq i \leq m}{f_{i,1}s_ps_p^*\bar{f}_{i,2}} = 1. 
\end{array}\]
The second claim follows from Lemma~\ref{lem:loc homeo right reconstruction}.
\end{proof}

\noindent Since finite open covers of the form appearing in Lemma~\ref{lem:CNP equiv} always exist for surjective local homeomorphisms of compact Hausdorff spaces, see Lemma~\ref{lem:loc homeo right reconstruction}, there are in fact functions $f_{i,j}$ satisfying the reconstruction formula for each $p \in P$. Thus, relation (IV) is non-void.

Let us recall the definition of the C*-algebra $\CO[G,P,\theta]$ for an irreversible algebraic dynamical system $(G,P,\theta)$ from \cite{Sta1} because there is a close connection between the defining relations: $\CO[G,P,\theta]$ is the universal C*-algebra generated by a unitary representation $(u_{g})_{g \in G}$ of the group $G$ and a representation $(s_{p})_{p \in P}$ of the semigroup $P$ by isometries subject to the relations:
\[\begin{array}{lrcl}
(\text{CNP }1) & s_{p}u_{g} &\hspace*{-2.5mm}=\hspace*{-2.5mm}& u_{\theta_{p}(g)}s_{p}\vspace*{2mm}\\
(\text{CNP }2) & s_{p}^{*}u_gs_{q} &\hspace*{-2.5mm}=\hspace*{-2.5mm}& \begin{cases} 
u_{g_1}s_{(p \wedge q)^{-1}q}s_{(p \wedge q)^{-1}p}^{*}u_{g_2}& \text{ if } g = \theta_p(g_1)\theta_q(g_2),\\ 0,& \text{ else.}\end{cases}\vspace*{2mm}\\
(\text{CNP }3) & 1 &\hspace*{-2.5mm}=\hspace*{-2.5mm}& \sum\limits_{[g] \in G/\theta_{p}(G)}{e_{g,p}} \hspace*{2mm}\text{ if } [G : \theta_{p}(G)]< \infty,
\end{array}\]
where $e_{g,p} = u_{g}s_{p}s_{p}^{*}u_{g}^{*}$. 

We will now show that the two constructions yield the same C*-algebra if both methods are applicable, that is, if $(G,P,\theta)$ of finite type and $G$ is commutative, see Corollary~\ref{cor:CIAD vs SIDoFT}. Recall that the dual model $(\hat{G},P,\hat{\theta})$ is an irreversible $*$-commutative dynamical system of finite type in this case.

\begin{proposition}\label{prop:consistent constr for CIADoFT}
Let $(G,P,\theta)$ be a commutative irreversible algebraic dynamical system of finite type. If $(u_g)_{g \in G}$ and $(s_p)_{p \in P}$ denote the canonical generators of $\CO[G,P,\theta]$ and $(w_g)_{g \in G}$ and $(v_p)_{p \in P}$ denote the canonical generators of $\CO[\hat{G},P,\hat{\theta}]$, then 
\[\begin{array}{rcl}
\CO[G,P,\theta] &\stackrel{\varphi}{\longrightarrow}& \CO[\hat{G},P,\hat{\theta}]\\
u_gs_p &\mapsto& w_gv_p
\end{array}\]
is an isomorphism.
\end{proposition}
\begin{proof}
It is clear that $(w_g)_{g \in G}$ and $(v_p)_{p \in P}$ satisfy (CNP 1). (CNP 3) follows from (IV) since we can easily check the reconstruction formula required in (IV) on each $w_g$ and note that $C(\hat{G})$ can be identified with the closed linear span of $(w_g)_{g \in G}$. It remains to prove (CNP 2), that is, 
\[v_p^*w_gv_q = \chi_{\theta_p(G)\theta_q(G)}(g)~w_{g_1}v_{(p \wedge q)^{-1}q}v_{(p \wedge q)^{-1}p}^*w_{g_2} \text{ for all } g \in G \text{ and } p,q \in P,\]
for $g = \theta_p(g_1)\theta_q(g_2)$, and $v_p^*w_gv_q = 0$ otherwise. The case $g \in \theta_p(G)\theta_q(G)$ follows in a straightforward manner from (I) and (III), so suppose $g \notin \theta_p(G)\theta_q(G)$. Since $(G,P,\theta)$ is of finite type, $\theta_{(p \wedge q)^{-1}p}$ and $\theta_{(p \wedge q)^{-1}q}$ are strongly independent. So we have $\theta_{(p \wedge q)^{-1}p}(G)\theta_{(p \wedge q)^{-1}q}(G) = G$ and hence 
\[g \notin \theta_{p \wedge q}(\theta_{(p \wedge q)^{-1}p}(G)\theta_{(p \wedge q)^{-1}q}(G)) = \theta_{p \wedge q}(G)\]
and, with the help of Example~\ref{ex:transfer operator disc ab fin type} we conclude that
\[v_p^*w_gv_q = v_{(p \wedge q)^{-1}p}^*v_{p \wedge q}^*w_gv_{p \wedge q}v_{(p \wedge q)^{-1}q} \stackrel{(II)}{=} v_{(p \wedge q)^{-1}p}^*L_{p \wedge q}(w_g)v_{(p \wedge q)^{-1}q} = 0.\]
Thus we have shown that $\varphi$ is a surjective $\ast$-homomorphism. In order to see that $\varphi$ is an isomorphism, it suffices to check that $C^*((u_g)_{g \in G}) \cong C(\hat{G})$ and $(s_p)_{p \in P}$ satisfy (I)--(IV). Condition (I) is nothing but (CNP 1). Conditions (II) and (III) follow from (CNP 2) using Example~\ref{ex:transfer operator disc ab fin type}. Finally, (IV) can be deduced from (CNP 3) with the help of Lemma~\ref{lem:CNP equiv}.
\end{proof}

\noindent We have seen in Lemma~\ref{lem:CNP equiv} that we can always choose elements $f_{i,j}$ satisfying the reconstruction formula for (IV) in such a way that we get a C*-algebraic partition of unity in $\CO[X,P,\theta]$, that is, the corresponding elements are positive and sum up to one. Unless $X$ is totally disconnected, this may produce a number of genuine positive elements exceeding the actual number of preimages a single point has. For example, the minimal number of elements appearing in a partition of unity as in Lemma~\ref{lem:CNP equiv} for the map $\times2: \IT \longrightarrow \IT$ is three. 

One particular feature of commutative irreversible algebraic dynamical systems of finite type compared to arbitrary irreversible $*$-commutative dynamical systems of finite type is that we can choose the elements satisfying the reconstruction formula for (IV) in a different manner using the algebraic structure. This allows us to reduce the number of positive elements needed to the optimal value, that is, the size of the kernel of the group endomorphism on $\hat{G}$. Moreover, the elements forming the C*-algebraic partition of unity are projections in this case. 

Now that we have already established some connections to Section~2.2, let us start with an analysis of basic properties of the C*-algebra $\CO[X,P,\theta]$. First of all, there is a natural representation of $\CO[X,P,\theta]$ on $\ell^2(X)$, whose standard orthonormal basis will be denoted by $(\xi_{x})_{x \in X}$:

\begin{proposition}\label{prop:elementary rep SIDoFT}
For $f \in C(X)$, let $M_{f} \xi_{x} = f(x) \xi_{x}$. For $p \in P$, define $S_{p} \xi_{x} = N_{p}^{-\frac{1}{2}} \sum_{y \in \theta_{p}^{-1}(x)}{\xi_{y}}$. Then the map $fs_p \mapsto M_fS_p$ defines a representation $\varphi$ of $\CO[X,P,\theta]$ on $\ell^{2}(X)$, which is faithful on $C(X)$.
\end{proposition} 
\begin{proof}
Firstly, $S_p^* (\xi_x) = N_p^{-\frac{1}{2}}\xi_{\theta_p(x)}$ for all $p \in P$ and $x \in X$ since 
\[\langle S_p^*(\xi_x),\xi_y \rangle = \langle \xi_x,S_p(\xi_y) \rangle = \chi_{\theta_p^{-1}(y)}(x)~N_p^{-\frac{1}{2}}.\] 
Thus, $S_p$ is an isometry. $(S_p)_{p \in P}$ is a representation of $P$ because 
\[\begin{array}{lcl}
S_pS_q(\xi_x) &=& N_{q}^{-\frac{1}{2}}\sum\limits_{y \in \theta_q^{-1}(x)} S_p(\xi_y)\\
&=& (N_pN_q)^{-\frac{1}{2}}\sum\limits_{\substack{y \in \theta_q^{-1}(x) \\ z \in \theta_p^{-1}(y)}} \xi_z\\
&=& (N_{pq})^{-\frac{1}{2}}\sum\limits_{z \in \theta_{pq}^{-1}(x)} \xi_z\\
&=& S_{pq}(\xi_x).
\end{array}\] 
\begin{enumerate}[(I)]
\item If $p$ and $q$ are relatively prime in $P$, then $\theta_p$ and $\theta_q$ $\ast$-commute according to Definition~\ref{def:SID}~(C). Using the equivalent condition (iii) from Proposition~\ref{prop: eq star ind}, we obtain
\[\begin{array}{c} S_p^{*}S_q (\xi_x) = N_{pq}^{-\frac{1}{2}} \sum\limits_{y \in \theta_p(\theta_q^{-1}(x))}\xi_y = N_{pq}^{-\frac{1}{2}} \sum\limits_{y \in \theta_q^{-1}(\theta_p(x))}\xi_y = S_qS_p^{*} (\xi_x), \end{array}\]
so $S_p$ and $S_q$ doubly commute. 
\item $S_pM_f = M_{\alpha_p(f)}S_p$ is readily verified for all $f \in C(X)$ and $p \in P$.
\item $S_p^{*}M_fS_p = M_{L_p(f)}$ is also straightforward.
\item For $\nu_i = (N_pv_i)^{\frac{1}{2}}$, where $(v_i)_{1 \leq i \leq n}$ is a partition of unity such that $\theta_p|_{\supp v_i}$ is injective for all $i$ (as in Lemma~\ref{lem:CNP equiv}), we compute
\[\begin{array}{lclcl} 
\sum\limits_{1 \leq i \leq n}M_{\nu_i}S_pS_p^*M_{\nu_i}(\xi_x) &=& \sum\limits_{1 \leq i \leq n}\sum\limits_{y \in \theta_p^{-1}(\theta_p(x))}\underbrace{(v_i(y)v_i(x))^{\frac{1}{2}}}_{\delta_{x~y}}~\xi_y \vspace*{2mm}\\
&=& \sum\limits_{1 \leq i \leq n} v_{i}(x)\xi_{x} \\
&=& \xi_x.
\end{array}\]
We infer from Lemma~\ref{lem:CNP equiv} that this yields (IV) since the proof provided there only uses the additional property (II), which we have already established for $S_p$ and $M_f$.
\end{enumerate}
Thus, $\varphi$ is a $\ast$-homomorphism by the universal property of $\CO[X,P,\theta]$ and it is clear that $\varphi$ is faithful on $C(X)$.
\end{proof}

\begin{lemma}\label{lem:O-alg SIDoFT - lin span}
The linear span of $\{fs_ps_q^*g \mid f,g \in C(X),p,q \in P\}$ is dense in $\CO[X,P,\theta]$.
\end{lemma}
\begin{proof}
The set is closed under taking adjoints and contains the generators, so we only have to show that it is multiplicatively closed. Let $p_i,q_i \in P,~f_i,g_i \in C(X)$ and $a_i := f_is_{p_i}s_{q_i}^*g_i$ for $i = 1,2$. Additionally, choose a partition of unity $(v_j)_{1 \leq j \leq n}$ subordinate to a finite open cover $(U_j)_{1 \leq j \leq n}$ of $X$ such that $\theta_{q_{1} \vee p_{2}}|_{U_{j}}$ is injective and $\nu_j := (N_{q_1 \vee p_2}~v_j)^{\frac{1}{2}}$ for all $j$. Then, we get 
\[\begin{array}{lcl}
a_{1}a_{2} \hspace*{-3mm}&\stackrel{\text{(IV)}}{=}&\hspace*{-3mm} a_{1}\sum\limits_{1 \leq j \leq n}{\nu_{j}s_{q_{1} \vee p_{2}}s_{q_{1} \vee p_{2}}^{*}\nu_{j}}a_{2} \\
&\stackrel{\text{(II)}}{=}&\hspace*{-3mm} \sum\limits_{1 \leq j \leq n}{f_{1}s_{p_{1}}L_{q_{1}}(g_{1}\nu_{j})s_{q_{1}^{-1}(q_{1} \vee p_{2})}s_{p_{2}^{-1}(q_{1} \vee p_{2})}^{*}L_{p_{2}}(\nu_{j}f_{2})s_{q_{2}}^{*}g_{2}} \\
&\stackrel{\text{(I)}}{=}&\hspace*{-3mm} \sum\limits_{1 \leq j \leq n}{f_{1}~\alpha_{p_{1}} \circ L_{q_{1}}(g_{1}\nu_{j})s_{p_{1}q_{1}^{-1}(q_{1} \vee p_{2})}s_{q_{2}p_{2}^{-1}(q_{1} \vee p_{2})}^{*}\alpha_{q_{2}} \circ L_{p_{2}}(\nu_{j}f_{2})g_{2}}.
\end{array}\] 
\end{proof}

\noindent The remainder of this section will deal with degrees of faithfulness of conditional expectations related to a core subalgebra of $\CO[X,P,\theta]$. Recall that the enveloping group $H = P^{-1}P$ of $P$ is discrete abelian. If we denote its Pontryagin dual by $L$, which is then a compact abelian group, we get a so-called gauge action $\gamma$ of $L$ on $\CO[X,P,\theta]$ by
\[ \gamma_\ell(f) = f \text{ and } \gamma_\ell(s_p) = \ell(p) s_p \text{ for } f \in C(X), p \in P  \text{ and } \ell \in L.\]
It is well-known that actions of this form are strongly continuous.

\begin{definition}\label{def:core SIDoFT}
The fixed point algebra $\CO[X,P,\theta]^\gamma$ for the gauge action $\gamma$, denoted by $\CF$, is called the \textit{core} of $\CO[X,P,\theta]$. In addition, let 
\[\CF_p := C^{*}\left(\{fs_ps_p^*g \mid f,g \in C(X)\}\right)\] 
denote the subalgebra of $\CF$ corresponding to $p \in P$.
\end{definition}

\begin{lemma}\label{lem:cond exp to F}
Let $\mu$ denote the normalized Haar measure of the compact abelian group $L$. Then $E_1(a):= \int_{\ell \in L} \gamma_\ell(a)~d\mu(\ell)$ defines a faithful conditional expectation $E_1: \CO[X,P,\theta] \longrightarrow \CF$.
\end{lemma}
\begin{proof}
If $a \in \CO[X,P,\theta]$ is positive and non-zero, then there is a state $\psi$ on $\CO[X,P,\theta]$ such that $\psi(a) = \|a\|$. Since $\gamma_\ell(a) \geq 0$ in $\CO[X,P,\theta]$, we have $\psi(\gamma_\ell(a)) \geq 0$ for all $\ell \in L$. Thus, $\psi(a) = \|a\| > 0$ together with strong continuity of $\gamma$ implies
\[\begin{array}{c}\psi(E_1(a)) = \int\limits_{\ell \in L}\psi(\gamma_\ell(x))~d\mu(\ell) > 0.\end{array}\]
\end{proof}

\begin{proposition}\label{prop:F as ind limit}
$\CF$ is the closed linear span of $(fs_ps_p^*g)_{f,g \in C(X),p \in P}$. Moreover, $\CF_p \subset \CF_q$ holds whenever $q \in pP$ and hence $\CF = \overline{\bigcup_{p \in P} \CF_p}$.
\end{proposition} 
\begin{proof}
Clearly, every element $fs_ps_p^*g$ is fixed by $\gamma$. Conversely, if $a \in \CF$, we can approximate $a$ by finite linear combinations of elements $f_is_{p_i}s_{q_i}^*g_i$ according to Lemma~\ref{lem:O-alg SIDoFT - lin span}. Relying on the conditional expectation $E_1$ from Lemma~\ref{lem:cond exp to F}, we know that it suffices to take those $f_is_{p_i}s_{q_i}^*g_i$ satisfying $p_i = q_i$. If $q \in pP$ holds true, then we can employ (IV) for $p^{-1}q$ to deduce $\CF_p \subset \CF_q$. The last claim is an immediate consequence of this.
\end{proof}

\noindent The next observation and its proof are based on \cite{EV}*{Proposition 7.9}.

\begin{proposition}\label{prop:pos elts in Fp}
For $p \in P$, the subalgebra $\CF_p$ of $\CF$ satisfies
\[\begin{array}{ccl}
\CF_p &=& \Span\{fs_ps_p^*g \mid f,g \in C(X)\}\\
&\text{and}\\
(\CF_{p})_{+} &=& \Span\{fs_ps_p^*\bar{f} \mid f \in C(X)\}.
\end{array}\]
\end{proposition} 
\begin{proof}
The right hand side of the first equation is multiplicatively closed as
\[f_1s_ps_p^*g_1~f_2s_ps_p^*g_2 \stackrel{(II),(I)}{=} f_1 E_p(g_1f_2)s_ps_p^*g_2.\]
Let $a \in \CF_p, \varepsilon > 0$ and choose $m \in \N, f_k,g_k \in C(X),1\leq k \leq m$ such that 
\[\begin{array}{c}\|\sum\limits_{k = 1}^{m}{f_ks_ps_p^*g_k} - a\| < \varepsilon. \end{array}\] 
Pick $(\nu_i)_{1 \leq i \leq n}$ coming from a suitable partition of unity of $X$ for $\theta_p$ as in Lemma~\ref{lem:loc homeo right reconstruction}. In other words, the family $(\nu_{i})_{1 \leq i \leq n}$ satisfies (IV) from Definition~\ref{def:O[X,P,theta]}. Then we obtain 
\[\begin{array}{lcl}
\sum\limits_{k = 1}^{m}{f_ks_ps_p^*g_k} &=& \sum\limits_{i = 1}^{n}\nu_is_ps_p^*\nu_i \sum\limits_{k = 1}^{m}f_ks_ps_p^*g_k \sum\limits_{j = 1}^{n}{\nu_js_ps_p^*\nu_j} \vspace*{2mm}\\
&=& \sum\limits_{i,j = 1}^{n}{h_{i,j}~\nu_is_ps_p^*\nu_j},
\end{array}\]
where $h_{i,j} = \sum_{1 \leq k \leq m} E_p(\nu_i f_k)E_p(g_k\nu_j)$, so $n^{2}$ summands suffice to approximate $a$ up to $\varepsilon$.

For the second part, let $a \in (\CF_p)_+$. Then $a = b^*b$ holds for some $b \in \CF_p$. From the first part, we know that $b = \sum_{i = 1}^m f_is_ps_p^{*}g_i$ for some $m \in \N$ and suitable $f_i,g_i \in C(X)$. Therefore,
\[\begin{array}{c} a =  \sum\limits_{i,j = 1}^m \bar{g}_i s_ps_p^*\bar{f}_if_js_ps_p^*g_j = \sum\limits_{i,j = 1}^m \bar{g}_iE_p(\bar{f}_if_j)s_ps_p^*g_j. \end{array}\]
Recall that $E_p: C(X) \longrightarrow \alpha_p(C(X))$ is a conditional expectation and hence completely positive, see \cite{BO}*{Theorem 5.9}. Thus $(\bar{f}_if_j)_{i,j} \in M_m(C(X))_+$ implies that there exists some $c = (c_{ij})_{1 \leq i,j \leq m} \in M_m(E_p(C(X)))$ satisfying $\left(E_p(\bar{f}_if_j))\right)_{1 \leq i,j \leq m} = c^{*}c$. Setting $h_k = \sum_{i = 1}^m \overline{\alpha_p(c_{ki})g_i}$, we get 
\[\begin{array}{lcl}
a &=& \sum\limits_{i,j = 1}^m \bar{g}_i\left(\sum\limits_{k = 1}^m \bar{c}_{k i}c_{k j}\right) s_ps_p^*g_j\\
&=& \sum\limits_{k = 1}^m\left(\sum\limits_{i = 1}^m \overline{g_ic}_{k i}\right) s_ps_p^*\left(\sum\limits_{j = 1}^m c_{k j} g_j\right)\\
&=& \sum\limits_{k = 1}^m h_ks_ps_p^*\bar{h}_k.
\end{array}\]
\end{proof}

\noindent We need some results related to finite index endomorphisms. Since we do not assume that the reader is familiar with this notion, we shall recall it briefly and state the required results without proofs from \cite{Exe2}:

\begin{definition}[\cite{Wat}*{1.2.2,2.1.6}, \cite{Exe2}*{8.1}]\label{def:fin index}
Let $A$ be a C*-algebra. A pair $(\alpha,E)$ consisting of a $\ast$-endomorphism $\alpha$ of $A$ and a conditional expectation $A \stackrel{E}{\longrightarrow} \alpha(A)$ is said to be a \textit{finite-index endomorphism}, if there are $\nu_1,\dots,\nu_n \in A$ such that 
\[\begin{array}{c} \sum\limits_{1 \leq i \leq n}{\nu_iE(\nu_i^*a)} = a \text{ for all } a \in A.\end{array}\] 
\end{definition}

\begin{remark}
Concrete examples of this situation are provided by regular surjective local homeomorphisms $\theta$ of compact Hausdorff spaces $X$, see Section~1.4. In this case, we have $A = C(X)$, $\alpha(f)(x) = f(\theta(x))$ and $E = \alpha \circ L$, where $L$ is the natural transfer operator constructed in Example~\ref{ex:transfer operator for reg trafo}. To see this, observe that the requirement in Definition~\ref{def:fin index} is nothing but the reconstruction formula established in Lemma~\ref{lem:loc homeo right reconstruction}. From this perspective, finite-index endomorphisms can be thought of as irreversible C*-dynamical systems $(A,\alpha,E)$ that admit a finite Parseval frame.
\end{remark}

\noindent The following proposition is a reformulation of some results from \cite{Exe2} in terms of the terminology used within this exposition.   

\begin{proposition}[\cite{Exe2}*{8.6,8.8}]\label{prop:cond exp from F to C(X)}
The map $E_2:\CF \longrightarrow C(X)$ given by $fs_ps_p^*g \mapsto N_p^{-1}~fg$ is a conditional expectation. Moreover, it is the only conditional expectation from $\CF$ to $C(X)$ as the latter is commutative. 
\end{proposition}



\begin{corollary}\label{cor:cond exp to C(X)}
The map $G := E_2 \circ E_1$ is a conditional expectation from $\CO[X,P,\theta]$ to $C(X)$, whose restriction to $\CF_p$ is faithful for all $p \in P$.
\end{corollary}
\begin{proof}
By Proposition~\ref{prop:pos elts in Fp}, every element $a \in (\CF_p)_+$ is of the form $a = \sum_{i=1}^n f_is_ps_p^*\bar{f}_i$ for suitable $n \in \N$ and $f_i \in C(X)$. Then 
\[\begin{array}{c}  0 = G(a) = N_{p}^{-1} \sum\limits_{i=1}^n |f_i|^2 \end{array}\] 
implies $f_i = 0$ for all $i$, so $a = 0$. Thus $G$ is faithful on $\CF_p$. 
\end{proof}

\noindent Although the conditional expectation $G$ from Corollary~\ref{cor:cond exp to C(X)} may fail to be faithful, it satisfies the following weaker condition, which turns out to be useful in the proof of the main result Theorem~\ref{thm:top free char}.

\begin{lemma}\label{lem:weak faithfulness}
If $a \in \CO[X,P,\theta]_{+}$ satisfies $G(bab^{*}) = 0$ for all $b \in \CF$, then $a = 0$.
\end{lemma}
\begin{proof}
Let us assume $a \in \CF$ at first and suppose $G(bab^{*}) = 0$ holds for all $b \in \CF$. This implies $G(bac) = 0$ for all $b,c \in \CF$ as 
\[|G(bac)| \leq G(bacc^{*}ab^{*})^{\frac{1}{2}} \leq \|a^{\frac{1}{2}}c\| G(bab^{*})^{\frac{1}{2}} = 0.\]
For $a \neq 0$, $I := \{d \in \CF \mid G(bdc) = 0 \text{ for all } b,c \in \CF\}$ is a non-trivial ideal in $\CF$. By $\CF = \overline{\bigcup_{p \in P}\CF_{p}}$, see Proposition~\ref{prop:F as ind limit}, it follows that $I \cap \CF_{p} \neq 0$ for some $p \in P$, so there is some $d \in (\CF_{p})_{+} \setminus \{0\}$ such that $G(d) = 0$.
But Proposition~\ref{prop:pos elts in Fp} shows that $d = \sum_{i = 1}^{n}{f_is_ps_p^*\bar{f}_i}$ for some $n \in \IN$ and suitable $f_{i} \in C(X)$, so $0 = G(b) = N_{p}^{-1}~\sum_{i = 1}^{n}{|f_{i}|^{2}} \neq 0$ yields a contradiction. Thus, we conclude that, for $a \in \CF_+$, $G(bab^*) = 0$ for all $b \in \CF$ implies $a = 0$. Now let $a \in \CO[X,P,\theta]_{+}$ be arbitrary. Then 
\[0 = G(bab^{*}) = G \circ E_{1}(bab^{*}) = G(bE_{1}(a)b^{*}) \text{ for all } b \in \CF,\]
so $E_{1}(a) = 0$ by what we have just shown. But this forces $a = 0$ since $E_{1}$ is faithful according to Lemma~\ref{lem:cond exp to F}.
\end{proof}

\section{An alternative approach via product systems of Hilbert bimodules}
\noindent
This section provides a different perspective on the C*-algebra $\CO[X,P,\theta]$ from Definition~\ref{def:O[X,P,theta]}: It can be thought of as the Cuntz-Nica-Pimsner algebra of a product system of Hilbert bimodules $\CX$ naturally associated to the irreversible $*$-commutative dynamical system $(X,P,\theta)$, see Proposition~\ref{prop:PS for an SIDoFT} and Theorem~\ref{thm:isom ad-hoc PS for SIDoFT}. This identification is not obvious as the latter C*-algebra is defined as a universal object for Cuntz-Nica-Pimsner covariant representations of the product system. We start with a brief introduction to discrete product systems of Hilbert bimodules, their representation theory and associated C*-algebras.\vspace*{3mm} 

\addsubsection{Background on product systems of Hilbert bimodules}~\\
Unless specified otherwise, let $A$ be a unital C*-algebra and $P$ a discrete, left cancellative, commutative monoid with unit $1_P$. There is a natural partial order on $P$ defined by $p \leq q$ if $q \in pP$ and we will assume $P$ to be lattice-ordered with respect to this partial order. That is to say, for $p,q \in P$ there exists a unique least common upper bound $p \vee q \in P$. Hence, there is also a unique greatest common lower bound $p \wedge q = (p \vee q)^{-1}pq$ for $p$ and $q$. In particular, this condition forces $P^* = \{1_P\}$. We point out that all these requirements are satisfied for countably generated, free abelian monoids.

\begin{definition}\label{def:Hilbert bimodule}
A $\C$-vector space $\CH$ equipped with a right $A$-module structure and a bilinear map $\langle \cdot,\cdot \rangle: \CH \times \CH \longrightarrow A$, which is linear in the second component, is called a \textit{right pre-Hilbert $A$-module}, if the following relations are satisfied for all $\xi,\theta \in \CH$ and $a \in A$:
\[\begin{array}{llclrlcl}
(1) & \left\langle \xi,\theta.a\right\rangle &=& \left\langle \xi,\theta \right\rangle a&\hspace*{12mm} (2) & \left\langle \xi,\theta\right\rangle^{\ast} &=& \left\langle \theta,\xi \right\rangle \vspace*{2mm}\\
(3) & \left\langle \xi,\xi\right\rangle &\geq& 0&\hspace*{12mm} (4) & \left\langle \xi,\xi\right\rangle &=& 0 \Longleftrightarrow \xi = 0 
\end{array}\] 
A right pre-Hilbert $A$-module $\CH$ is said to be a \textit{right Hilbert $A$-module} if it is complete with respect to the norm $\|\xi\| = \|\left\langle \xi,\xi\right\rangle\|_{A}^{\frac{1}{2}}$. $\CH$ is called a \textit{Hilbert bimodule} over $A$ if, in addition, there is a left action of $A$ given by a $\ast$-homomorphism $\phi_{\CH}:A \longrightarrow \CL(\CH)$, where $\CL(\CH)$ denotes the C*-algebra of all adjointable linear operators from $\CH$ to $\CH$.
\end{definition}

\begin{example}\label{ex:Hilbert bimodules}
Let $X$ be a compact Hausdorff space, $\theta:X \longrightarrow X$ a regular surjective local homeomorphism for which the induced injective $*$-endomorphism of $C(X)$ is denoted by $\alpha$. Then we can construct a Hilbert bimodule $\CH = \hspace*{1mm}_{\id}\hspace*{-0.5mm}C(X)_{\alpha}$ over $C(X)$ as follows: Starting with $C(X)$, we define an inner product $\langle f,g \rangle := L(\bar{f}g)$ for all $f,g \in C(X)$, where $L$ denotes the transfer operator for $\alpha$, see Example~\ref{ex:transfer operator for reg trafo}. It is clear that $f \mapsto \|L(|f|^2)\|^{\frac{1}{2}}$ is actually a norm on $C(X)$. Due to \cite{LR}*{Lemma 3.3}, this norm is equivalent to the standard norm $\|\cdot\|_\infty$. Hence, $C(X)$ is already complete with respect to this norm. The left action is given by multiplication whereas the right action is defined as $f.g = f\alpha(g)$ for $f,g \in C(X)$.
\end{example}

\begin{definition}\label{def:gen comp op}
Let $\CH$ be a right Hilbert module over $A$. For $\xi,\eta \in \CH$, $\Theta_{\xi,\eta} \in \CL(\CH)$, given by $\Theta_{\xi,\eta}(\zeta) = \xi.\left\langle \eta, \zeta \right\rangle$ for $\zeta \in \CH$, is said to be a \textit{generalized rank one operator}. The closed linear span of $\left(\Theta_{\xi,\eta}\right)_{\xi,\eta \in \CH}$ inside $\CL(\CH)$ is called the \textit{C*-algebra of generalized compact operators} $\CK(\CH)$.
\end{definition}

\noindent It is clear that $\CK(\CH)$ is an ideal in $\CL(\CH)$. Suppose $\CH_1$ and $\CH_2$ are Hilbert bimodules over $A$ whose left and right actions are denoted by $\phi_1,\phi_2$ and $\rho_1,\rho_2$, respectively. Then 
\[\langle [\xi_1 \otimes \xi_2],[\eta_1 \otimes \eta_2] \rangle_{\CH_1 \otimes_A \CH_2} = \langle \xi_2,\phi_2(\langle \xi_1,\eta_1 \rangle_1)\eta_2\rangle_2\]
defines an inner product on $(\CH_1 \odot \CH_2)/\sim$, where $\xi_1 \otimes \xi_2 \sim \eta_1 \otimes \eta_2$ if there exists $a \in A$ such that $\xi_2 = \phi_2(a)\eta_2 \text{ and } \eta_1 = \xi_1\rho_1(a)$. The completion of $(\CH_1 \odot \CH_2)/\sim$ with respect to the norm induced by this inner product can be equipped with left and right actions induced from $\phi_1$ and $\rho_2$, respectively, yielding a Hilbert bimodule $\CH_1 \otimes_A \CH_2$. This Hilbert bimodule is called the balanced tensor product of $\CH_1$ and $\CH_2$ over $A$, see for instance \cite{Lan}*{Proposition 4.5}.

\begin{definition}\label{def:prod system}
A \textit{product system of Hilbert bimodules} over $P$ with coefficients in the C*-algebra $A$ is a monoid $\CX$ together with a monoidal homomorphism $\rho:\CX \longrightarrow P$ such that:
\begin{enumerate}[(1)]
\item $\CX_{p} := \rho^{-1}(p)$ is a Hilbert bimodule over $A$ for each $p \in P$,
\item $\CX_{1_P} \cong \hspace*{1mm}_{\id}\hspace*{-0.5mm}A_{\id}$ as Hilbert bimodules and 
\item for all $p,q \in P$, we have $\CX_{p} \otimes_{A} \CX_{q} \cong \CX_{pq}$ if $p \neq 1_P$, and $\CX_{1_P} \otimes_{A} \CX_{q} \cong \overline{\phi_{q}(A)\CX_{q}}$.
\end{enumerate}   
\end{definition}

\begin{remark}\label{rem:prod system}
The multiplicative structure of $\CX$ yields $\ast$-homomorphisms 
\[\begin{array}{rcl}
\CL(\CX_{p}) &\stackrel{\iota_{p}^{pq}}{\longrightarrow}& \CL(\CX_{pq})\\
T &\mapsto& T \otimes id_{\CX_{q}}
\end{array}\] 
for all $p,q \in P$, where we have identified $\CX_{p} \otimes_{A} \CX_{q}$ with $\CX_{pq}$. It is clear that $\iota_{p}^{p}$ is an isomorphism whereas $\iota_{1_P}^{p}$ is an isomorphism if and only if $\CX_{p}$ is essential, that is, $\overline{\phi_{p}(A)\CX_{p}} = \CX_p$. If $A$ is unital, this is equivalent to $\phi_{p}(1_{A}) = 1_{\CL(\CX_p)}$.
\end{remark}


\begin{example}\label{ex:gen comp op}
The maps $\iota_{p}^{pq}$ introduced in \ref{rem:prod system}~c) need not map generalized compact operators to generalized compact operators. Consider for example the trivial case of $A = \C$ acting by multiplication on the fibers $\CX_{p} = H$, where $H$ is a separable, infinite-dimensional Hilbert space (equipped with a suitable product structure obtained from bijections $\N^{2} \longrightarrow \N$): $\iota_{1_P}^{p}$ is determined by the projection $\iota_{1_P}^{p}(1) = 1_{\CL(\CX_{p})}$ which is infinite and hence non-compact.      
\end{example}

\noindent There is a less restrictive requirement called compact alignment, which has been introduced for product systems over quasi-lattice ordered groups to avoid a certain pathology for the representation theory of product systems, see \cite{Fow1}*{Example 1.3}. 


\begin{definition}\label{def:cp aligned}
A product system of Hilbert bimodules $\CX$ over $P$ is called \textit{compactly aligned}, if 
\[\iota_{p}^{p \vee q}(k_{p})\iota_{q}^{p \vee q}(k_{q}) = (k_p{\otimes}1_{\CL(\CX_{(p \wedge q)^{-1}q})})(k_q{\otimes}1_{\CL(\CX_{(p \wedge q)^{-1}p})}) \in \CK(\CX_{p \vee q})\]
holds for all $p,q \in P$ and $k_{p} \in \CK(\CX_{p}), k_{q} \in \CK(\CX_{q})$.
\end{definition}

\noindent We will now proceed with stronger regularity properties, namely coherent systems of finite Parseval frames or orthonormal bases for product systems of Hilbert bimodules. This concept has been studied to some extent in \cite{HLS}.

\begin{definition}\label{def:ONB HB}
Let $\CH$ be a Hilbert bimodule over $A$ and $(\xi_{i})_{i \in I} \subset \CH$. Consider the following properties:
\begin{enumerate}[(1)]
\item $\left\langle \xi_{i},\xi_{j} \right\rangle = \delta_{ij}1_{A}$ for all $i,j \in I$.
\item $\eta = \sum\limits_{i \in I}{\xi_{i}\left\langle \xi_{i},\eta \right\rangle}$ for all $\eta \in \CH$.  
\end{enumerate}
If the family $(\xi_{i})_{i \in I}$ satisfies (2), it is called a \textit{Parseval frame} for $\CH$. A Parseval frame is said to be an \textit{orthonormal basis} for $\CH$, if it satisfies (1).
\end{definition}

\noindent A Parseval frame (orthonormal basis) is called finite if it consists of finitely many elements. In contrast to the case of orthonormal bases of a Hilbert space, the cardinality of an orthonormal basis of a Hilbert bimodule is not an invariant of the bimodule. For example, take $A = C\left([-2,-1] \cup [1,2]\right)$ and let $\CH = \hspace*{1mm}_{id}\hspace*{-0.5mm}A_{id}$ as in Example~\ref{ex:Hilbert bimodules}~(b). Then $\{1\}$ and $\{\chi_{[-2,-1]},\chi_{[1,2]}\}$ are both orthonormal bases for $\CH$.

\begin{lemma}\label{lem:ONB gives matrix units}
Let $\CH$ be a Hilbert bimodule and $(\xi_{i})_{i \in I} \subset \CH$. If equation~\ref{def:ONB HB}~(1) holds, then $\left(\Theta_{\xi_{i},\xi_{j}}\right)_{i,j \in I}$ is a system of matrix units. If $(\xi_{i})_{i \in I}$ is a finite Parseval frame, then $\sum_{i = 1}^{n}{\Theta_{\xi_{i},\xi_{i}}} = 1_{\CL(\CH)}$ and $\CK(\CH) = \CL(\CH)$.
\end{lemma}
\begin{proof}
\ref{def:ONB HB}~(1) directly implies that $\left(\Theta_{\xi_i,\xi_j}\right)_{i,j \in I}$ is a system of matrix units. The reconstruction formula \ref{def:ONB HB}~(2) shows that $\left(\sum_{i \in F}\Theta_{\xi_i,\xi_i}\right)_{F \subset I \text{ finite}}$ converges strongly to $1_{\CL(\CH)}$. Thus, if $I$ is finite, we have $\sum_{i = 1}^n{\Theta_{\xi_i,\xi_i}} = 1_{\CL(\CH)}$ and the last claim follows since $\CK(\CH)$ is an ideal in $\CL(\CH)$.
\end{proof}

\begin{remark}\label{rem:ONBs tensored}
A useful aspect of Parseval frames of Hilbert bimodules is that they are well-behaved with respect to the balanced tensor product: If $\CH_1$ and $\CH_2$ are Hilbert bimodules over $A$ with Parseval frames $(\xi_i)_{i \in I}$ and $(\eta_j)_{j \in J}$, respectively, then $(\xi_i \otimes \eta_j)_{(i,j) \in I{\times}J}$ is a Parseval frame for $\CH_1 \otimes_A \CH_2$, see \cite{LR}*{Lemma 4.3} for a detailed proof. Therefore, a product system $\CX$ of Hilbert bimodules over $P$ is a product system with Parseval frames if and only if $\CX_p$ admits a Parseval frame for each irreducible $p \in P$. Here $p \in P$ is said to be irreducible if $p = qr$ for $q,r \in P$ implies $q = 1_P$ or $r = 1_P$. The same statements hold for orthonormal bases instead of Parseval frames.
\end{remark}


\begin{remark}\label{rem:Parseval fr from pou}
Suppose $X$ is a compact Hausdorff space and $\theta_1,\theta_2:X \longrightarrow X$ are commuting regular surjective local homeomorphisms with $|\theta_1^{-1}(x)| = N_1$ and $|\theta_2^{-1}(x)| = N_2$ (where $x \in X$ is arbitrary). For $i = 1,2$, denote by $\alpha_i$ the endomorphism of $C(X)$ given by $f \mapsto f \circ \theta_i$. As in Lemma~\ref{lem:loc homeo right reconstruction}, let us choose partitions of unity $(v_{1,i})_{i \in I_1}$ and $(v_{2,i})_{i \in I_2}$ subordinate to finite open covers $\CU_1 = (U_{1,i})_{i \in I_1}$ and $\CU_2 = (U_{2,i})_{i \in I_2}$ of $X$ for $\theta_1$ and $\theta_2$, respectively. By Lemma~\ref{lem:loc homeo right reconstruction}, each of these partitions of unity gives rise to a Parseval frame $(\nu_{j,i_j})_{i_j \in I_j}$ with $\nu_{j,i_j} := (N_{j}v_{j,i_j})^{\frac{1}{2}}$ of the Hilbert bimodule $C(X)_{\alpha_j}$, which is equipped with the inner product coming from the transfer operator $L_j$ as constructed in Example~\ref{ex:transfer operator for reg trafo}. Taking into account \cite{LR}*{Lemma 4.3}, it is no surprise that $(\nu_{1,i})_{i \in I_1}$ and $(\nu_{2,i})_{i \in I_2}$ yield a Parseval frame on the balanced tensor product of the two modules, i.e. on $C(X)_{\alpha_1\alpha_2}$. Interestingly, Lemma~\ref{lem:refine cover and pou} indicates that this Parseval frame is again of the same form: We can construct a partition of unity $(v_{1,i_1}\alpha_1(v_{2,i_2}))_{i_1 \in I_1,i_2 \in I_2}$ for $X$ from $(v_{1,i})_{i \in I_1}$ and $(v_{2,i})_{i \in I_2}$ which fits into the picture of Lemma~\ref{lem:loc homeo right reconstruction} for $\theta_1\theta_2$. 
\end{remark}

\begin{definition}\label{def:ONB sys for prod sys}\label{def:ONB-types}
A product system of Hilbert bimodules $\CX$ over $P$ with coefficients in a unital C*-algebra $A$ is called a \textit{product system of finite type} if there exists a finite Parseval frame for $\CX_p$ for each irreducible $p \in P$. 
\end{definition}

\begin{remark}\label{rem:fin type is cp aligned}
If $\CX$ is a product system of finite type, then each fiber $\CX_p$ has a finite Parseval frame by applying Remark~\ref{rem:ONBs tensored} to a decomposition of $p$ into irreducible elements (with multiplicities). So Lemma~\ref{lem:ONB gives matrix units} implies that $\CX$ is compactly aligned whenever it is of finite type.
\end{remark}

\addsubsection{Representation theory and C*-algebras for product systems}~\\
In this part, we recall some elementary facts about the representation theory for product systems of Hilbert bimodules in order to present the construction of the Cuntz-Nica-Pimsner algebra for compactly aligned product systems of Hilbert bimodules.  

\begin{definition}\label{def:rep of prod sys}
Let $\CX$ be a product system over $P$ and suppose $B$ is a C*-algebra. A map $\CX \stackrel{\varphi}{\longrightarrow} B$, whose fiber maps $\CX_{p} \longrightarrow B$ are denoted by $\varphi_{p}$, is called a \textit{Toeplitz representation} of $\CX$, if:
\begin{enumerate}[(1)]
\item $\varphi_{1_P}$ is a $\ast$-homomorphism.
\item $\varphi_{p}$ is linear for all $p \in P$.
\item $\varphi_{p}(\xi)^{*}\varphi_{p}(\eta) = \varphi_{1_P}\left(\left\langle \xi,\eta\right\rangle\right)$ for all $p \in P$ and $\xi,\eta \in \CX_{p}$.
\item $\varphi_{p}(\xi)\varphi_{q}(\eta) = \varphi_{pq}(\xi \eta)$ for all $p,q \in P$ and $\xi \in \CX_{p}, \eta \in \CX_{q}$.
\end{enumerate}   	
\end{definition}

\noindent A Toeplitz representation will be called a representation whenever there is no ambiguity. 

\begin{remark}\label{rem:rep gives hom for cp op}
Let $\varphi$ be a representation of $\CX$ in $B$. For each $p \in P$, $\varphi$ induces a $\ast$-homomorphism $\psi_{\varphi,p}:\CK(\CX_{p}) \longrightarrow B$ given by $\Theta_{\xi,\eta} \mapsto \varphi_p(\xi)\varphi_p(\eta)^*$.
\end{remark}

\begin{lemma}\label{lem:rep contractive}
A representation $\varphi$ of $\CX$ in $B$ is contractive. $\varphi$ is isometric if and only if $\varphi_{1_P}$ is injective.
\end{lemma}
\begin{proof}
Given $p \in P, \xi \in \CX_{p}$, we get
\[\|\varphi_{p}(\xi)\|_B^{2} = \|\varphi_{p}(\xi)^{*}\varphi(\xi)\|_B \stackrel{(3)}{=} \| \varphi_{1_P}(\left\langle \xi,\xi\right\rangle)\|_B \stackrel{(1)}{\leq} \|\left\langle \xi,\xi\right\rangle\|_{A} = \|\xi\|_{X_{p}}^{2}.\]
Since $\varphi_{1_P}$ is a $\ast$-homomorphism, it is injective if and only if it is isometric. In this case the computation from above gives $\|\varphi_{p}(\xi)\|_B = \|\xi\|_{X_{p}}$.  
\end{proof}

\begin{definition}\label{def:Nica cov}
A representation $\varphi$ of a compactly aligned product system $\CX$ in $B$ is \textit{Nica covariant}, if 
\[\psi_{\varphi,p}(k_{p})\psi_{\varphi,q}(k_{q}) = \psi_{\varphi,p \vee q}\left(\iota_{p}^{p \vee q}(k_{p})\iota_{q}^{p \vee q}(k_{q})\right)\]
holds for all $p,q \in P$ and $k_{p} \in \CK(\CX_{p}),k_{q} \in \CK(\CX_{q})$.
\end{definition}

\noindent Note that compact alignment is needed to ensure that $\iota_{p}^{p \vee q}(k_{p})\iota_{q}^{p \vee q}(k_{q})$ is contained in the domain of $\psi_{\varphi,p \vee q}$. While Nica covariance is an outcome of having a product system instead of a single Hilbert bimodule and its form is rather straightforward, there have been different attempts to generalize the notion of Cuntz-Pimsner covariance from the case of a single Hilbert bimodule to general product systems. Let us recall the covariance condition introduced in \cite{Pim} for the corresponding product system over $\N$: Suppose $\CH$ is a Hilbert bimodule over a C*-algebra $A$ and $(\varphi_0,\varphi_1)$ is a representation of $\CH$. Then we can equally well study the induced representation $\varphi$ of the product system $\CX$ over $\N$ with fibers $\CX_n = \CH^{{\otimes}n}$, where $\CH^{{\otimes}0} = A$. $(\varphi_0,\varphi_1)$ is said to be Cuntz-Pimsner covariant, if $\varphi_0(a) = \psi_{\varphi,n}(\phi_n(a)) \text{ holds for all } a \in \phi_n^{-1}(\CK(\CX_n))$.

The intuitive approach to define a notion of Cuntz-Pimsner covariance for product systems by requiring Cuntz-Pimsner covariance on each fiber has been set up in \cite{Fow2}. In \cite{Kat}*{Definition 3.4}, Takeshi Katsura introduced a weaker version: Instead of $\phi_{p}^{-1}(\CK(\CX_{p}))$, only $\phi_{p}^{-1}(\CK(\CX_{p})) \cap (\ker\phi_p)^\perp$ is taken into account. Since the left actions in our examples will always be injective, we will not discuss this aspect any further. 

Several years later, a more involved approach of Aidan Sims and Trent Yeend led to a potentially different notion of Cuntz-Nica-Pimsner covariance, see \cite{SY}*{Section 3}. According to \cite{SY}, their definition is motivated by the study of graph C*-algebras and was expected to be more suitable in the case of product systems where the left action $\phi$ need not be injective.

We will now present both covariance conditions for product systems and indicate what is currently known about their connections as well as their relation to Nica covariance. In order to avoid technicalities, we restrict ourselves to the case where the left action $\phi_p$ on each fiber $\CX_p$ is injective. Therefore, we can neglect the inflation process from $\CX$ to $\tilde{\CX}$ taking place in \cite{SY}*{Section 3}. At this point, one may expect that the two notions ought to coincide. This is true at least to some extent, but non-trivial, see \cite{SY}*{Proposition 5.1 and Corollary 5.2}.   

\begin{definition}\label{def:CNP cov}
Let $B$ be a C*-algebra and suppose $\CX$ is a compactly aligned product system of Hilbert bimodules over $P$ with coefficients in $A$.
\[\begin{array}{cl}
(\text{CP}_{F}) &\hspace*{-2mm} \text{A representation $\CX \stackrel{\varphi}{\longrightarrow} B$ is called \textit{Cuntz-Pimsner covariant}}\\ 
&\hspace*{-2mm}\text{in the sense of \cite{Fow2}*{Section 1}, if it satisfies}\vspace*{2mm}\\ 
&\hspace*{-2mm}\psi_{\varphi,p}(\phi_{p}(a)) = \varphi_{1_P}(a) \text{ for all } p \in P \text{ and } a \in \phi_{p}^{-1}(\CK(\CX_{p})) \subset A.\vspace*{2mm}\\
(\text{CP}) &\hspace*{-2mm} \text{A representation $\CX \stackrel{\varphi}{\longrightarrow} B$ is called \textit{Cuntz-Pimsner covariant}}\\
&\hspace*{-2mm}\text{in the sense of \cite{SY}*{Definition 3.9}, if the following holds:}\\
&\hspace*{-2mm}\text{Suppose $F \subset P$ is finite and we fix $k_{p} \in \CK(\CX_{p})$ for each $p \in F$.}\\ 
&\hspace*{-2mm}\text{If, for every $r \in P$, there is $s \geq r$ such that}\vspace*{2mm}\\ 
&\hspace*{-2mm}\begin{array}{llcll}
&\sum\limits_{p \in F}{\iota_{p}^{t}(k_{p})} &=& 0 & \text{ holds for all $t \geq s$,}\vspace*{2mm}\\
\text{then }&\sum\limits_{p \in F}{\psi_{\varphi,p}(k_{p})} &=& 0 \text{ holds true.} 
\end{array}\\
(\text{CNP}) &\hspace*{-2mm} \text{A representation $\CX \stackrel{\varphi}{\longrightarrow} B$ is said to be \textit{Cuntz-Nica-Pimsner}}\\
&\hspace*{-2mm}\text{\textit{covariant}, if it is Nica covariant and $(\text{CP})$-covariant.}
\end{array}\]
\end{definition}

\noindent When Aidan Sims and Trent Yeend introduced their alternative notion of Cuntz-Pimsner covariance, they observed that it agrees with the one proposed by Neal J. Fowler in special cases, compare \cite{SY}*{Proposition 5.1}:

\begin{proposition}\label{prop:equiv of cov cond}
Suppose $\CX$ is a compactly aligned product system over $P$ with coefficients in a unital C*-algebra $A$ such that the left action $\phi_p$ on $\CX_p$ is injective for all $p \in P$. If a representation $\varphi$ of $\CX$ is $(CP_F)$-covariant, then it is $(CP)$-covariant. If the left action $\phi_p(A)$ is by compacts for all $p \in P$, then the converse holds as well.
\end{proposition}

\noindent In some instances, $(CP_F)$-covariance is known to imply Nica covariance. The result we are going to use is due to Fowler and we refer to \cite{Fow2}*{Proposition 5.4} for a proof.

\begin{proposition}\label{prop:CP-F cov gives Nica cov}
If $\CX$ is a compactly aligned product system over $P$ with coefficients in a unital C*-algebra $A$ such that $\phi_p(1_{A}) = 1_{\CL(\CX_p)} \in \CK(\CX_p)$ for all $p \in P$, then every ($CP_{F}$)-covariant representation is also Nica covariant.
\end{proposition}

\begin{corollary}\label{cor:equiv of cov cond - fin type}
If $\CX$ is a product system of finite type, then a representation $\varphi$ of $\CX$ is $(\text{CNP})$-covariant if and only if it is $(CP_F)$-covariant.
\end{corollary}
\begin{proof}
The result follows from Lemma~\ref{lem:ONB gives matrix units} together with Proposition~\ref{prop:equiv of cov cond} and Proposition~\ref{prop:CP-F cov gives Nica cov}.
\end{proof}


\begin{definition}\label{def:alg for prod sys}
For a compactly aligned product system $\CX$ over $P$ define $\CT_{\CX}$ to be the C*-algebra given by a Toeplitz representation $\iota_{\CT_{\CX}}$ of $\CX$ that is universal for Toeplitz representations. 
Similarly, $\mathcal{NT}_{\CX}$ and $\CO_{\CX}$ are the C*-algebras given by a universal Nica-covariant representation $\iota_{\mathcal{NT}_{\CX}}$ and a universal Cuntz-Nica-Pimsner covariant representation $\iota_{\CO_{\CX}}$, respectively. $\CT_{\CX},~ \mathcal{NT}_{\CX}$ and $\CO_{\CX}$ are called the \textit{Toeplitz algebra}, the \textit{Nica-Toeplitz algebra}, and the \textit{Cuntz-Nica-Pimsner algebra} associated to $\CX$.
\end{definition}

\addsubsection{\texorpdfstring{The case of irreversible $*$-commutative dynamical systems}{The case of irreversible *-commutative dynamical systems}}~\\
We will now show how to treat irreversible $*$-commutative dynamical systems in the framework of product systems of Hilbert bimodules.

\begin{proposition}\label{prop:PS for an SIDoFT}
Suppose $(X,P,\theta)$ is an irreversible $*$-commutative dynamical system of finite type and $P \stackrel{\alpha}{\curvearrowright} C(X)$ is the action induced by $\theta$, i.e. $\alpha_p(f) = f \circ \theta_p$ for $p \in P$ and $f \in C(X)$. Then $\CX_p := C(X)_{\alpha_p}$, with left action $\phi_p$ given by multiplication in $C(X)$ and inner product $\langle f,g \rangle_p = L_p(\overline{f}g)$ is an essential Hilbert bimodule, where $L_p$ is the natural transfer operator associated to $\alpha_p$, see Example~\ref{ex:transfer operator for reg trafo}. The disjoint union of all $\CX_p, p \in P$ forms a product system $\CX$ of finite type over $P$ with coefficients in $C(X)$.
\end{proposition} 
\begin{proof}
To see that $\CX_p$ is an essential Hilbert bimodule, we recall that the transfer operator $L_p$, given by $L_p(f)(x) = \frac{1}{N_p} \sum_{y \in \theta_p^{-1}(x)}\overline{f}(y)$,
is a positive, linear map such that $L_p(f\alpha_p(g)) = L_p(f)g$ holds for all $f,g \in C(X)$. Thus, we can use \cite{LR}*{Lemma 3.3} to conclude that the seminorm $\|f\|_p := \langle f,f\rangle_p^{\frac{1}{2}}$  on $C(X)$ is equivalent to $\|\cdot\|_{\infty}$. Thus, $\langle\cdot,\cdot\rangle$ is positive definite on $C(X)$ and $C(X)$ is complete with respect to $\|\cdot\|_p$. The $\CX_p$ form a product system since 
\[\begin{array}{cllcl}
\CX_p \hspace*{-2mm}&\otimes_{C(X)} &\hspace*{-2mm} \CX_q & \stackrel{M_{p,q}}{\longrightarrow}& \CX_{pq}\\
f \hspace*{-2mm}&\otimes&\hspace*{-2mm} g &\mapsto & f\alpha_p(g)
\end{array}\]
defines an isomorphism of Hilbert bimodules. Indeed, the left action is the same on both sides and 
\[M_{p,q}((f \otimes g).h) = M_{p,q}(f \otimes g\alpha_q(h)) = f\alpha_p(g)\alpha_{pq}(h) = M_{p,q}(f \otimes g).h\]
shows that the right actions match. Finally, the inner products coincide as
\[\begin{array}{lcl}
\langle M_{p,q}(f \otimes g),M_{p,q}(f' \otimes g')\rangle_{pq} &=& \langle f\alpha_p(g),f'\alpha_p(g')\rangle_{pq}\vspace*{2mm}\\ 
&=& L_{pq}\left(\alpha_p(\overline{g})\overline{f}f'\alpha_p(g')\right)\vspace*{2mm}\\
&=& L_{q}\left(\overline{g}L_p(\overline{f}f')g'\right)\vspace*{2mm}\\
&=& \langle g,\phi_q\left(\langle f,f'\rangle_p\right)g'\rangle_q \vspace*{2mm}\\
&=& \langle f \otimes g,f' \otimes g'\rangle_{\CX_p \otimes_{C(X)} \CX_q}.
\end{array}\] 
This shows that we have an injective morphism of Hilbert bimodules. Due to the structure of the balanced tensor product, $f \otimes g = f\alpha_p(g) \otimes 1_{C(X)}$ and $\alpha_p(1_{C(X)}) = 1_{C(X)}$, so $M_{p,q}$ is surjective as well. Thus, $\CX$ is a product system over $P$ with coefficients in $C(X)$. Lastly, $\CX$ is seen to be of finite type by appealing to Lemma~\ref{lem:loc homeo right reconstruction}.
\end{proof}

\begin{theorem}\label{thm:isom ad-hoc PS for SIDoFT}
Suppose $(X,P,\theta)$ is an irreversible $*$-commutative dynamical system of finite type, let $\CX$ denote the product system constructed in Proposition~\ref{prop:PS for an SIDoFT}. Then the map 
\[\begin{array}{rcl}
\CO[X,P,\theta] &\stackrel{\varphi}{\longrightarrow}& \CO_{\CX}\\
fs_p &\mapsto& \iota_{\CO_{\CX},p}(f)
\end{array} \]
is an isomorphism.
\end{theorem}
\begin{proof}
The idea is to exploit universal properties on both sides. We begin by showing that $(\iota_{\CO_\CX,p}(1))_{p \in P}$ and $\iota_{\CO_\CX,1_P}(C(X))$ induce $\varphi$. First of all, note that 
\[\iota_{\CO_\CX,p}(1)^*\iota_{\CO_\CX,p}(1) = \iota_{\CO_\CX,1_P}(\langle 1,1 \rangle_p) = \iota_{\CO_\CX,1_P}(L_p(1)) = 1_{\CO_\CX}\]
and $\iota_{\CO_\CX,1_P}$ is a unital $\ast$-homomorphism. Conditions (I),(II) and (IV) are immediate:
\[\begin{array}{lrclcl}
\text{(I)}\hspace*{-2mm}& \iota_{\CO_\CX,1_P}(\alpha_p(f))\iota_{\CO_\CX,p}(1) &\hspace*{-2.5mm}=\hspace*{-2.5mm}& \iota_{\CO_\CX,p}(\alpha_p(f)) &\hspace*{-2.5mm}=\hspace*{-2.5mm}& \iota_{\CO_\CX,p}(1)\iota_{\CO_\CX,1_P}(f)\vspace*{2mm}\\
\text{(II)}\hspace*{-2mm}& \iota_{\CO_\CX,p}(1)^*\iota_{\CO_\CX,1_P}(f)\iota_{\CO_\CX,p}(1) &\hspace*{-2.5mm}=\hspace*{-2.5mm}& \iota_{\CO_\CX,1_P}(\langle 1,f \rangle_p) &\hspace*{-2.5mm}=\hspace*{-2.5mm}& \iota_{\CO_\CX,1_P}(L_p(f))\vspace*{2mm}\\
\text{(IV)}\hspace*{-2mm}&\multicolumn{5}{l}{\text{Whenever $f_{i,j} \in C(X)$, where $i = 1,\dots,n$ and $j = 1,2$,}}\\
&\multicolumn{5}{l}{\text{satisfy the reconstruction formula for $p \in P$, then}}\\
\hspace*{-2mm}&\multicolumn{5}{l}{\sum\limits_{1 \leq i \leq n} \iota_{\CO_\CX,1_P}(f_{i,1})\iota_{\CO_\CX,p}(1)\iota_{\CO_\CX,p}(1)^*\iota_{\CO_\CX,1_P}(f_{i,2})^*}\vspace*{2mm}\\
&\hspace*{2mm}= \psi_{\iota_{\CO_\CX},p}(\sum\limits_{1 \leq i \leq n}\Theta_{f_{i,1},f_{i,2}})&\hspace*{-2.5mm}=\hspace*{-2.5mm}& \psi_{\iota_{\CO_\CX},p}(\phi_p(1))
&\hspace*{-2.5mm}=\hspace*{-2.5mm}& \iota_{\CO_\CX,1_P}(1) = 1_{\CO_\CX}
\end{array}\]
by $(CP_F)$-covariance of $\iota_{\CO_\CX}$, see Definition~\ref{def:CNP cov} and Corollary~\ref{cor:equiv of cov cond - fin type}. Proving (III) is substantially harder. We need to show that the isometries corresponding to relatively prime $p,q \in P$ are doubly commuting. Since $\iota_{\CO_\CX,p}(1)$ and $\iota_{\CO_\CX,q}(1)$ are isometries, (III) is equivalent to 
\[\psi_{\iota_{\CO_\CX},p}(\Theta_{1,1})\psi_{\iota_{\CO_\CX},q}(\Theta_{1,1}) = \psi_{\iota_{\CO_\CX},pq}(\Theta_{1,1}).\]
Nica covariance of $\iota_{\CO_\CX}$ implies that this is in turn the same as
\[\psi_{\iota_{\CO_\CX},pq}\left(\iota_p^{pq}(\Theta_{1,1})\iota_q^{pq}(\Theta_{1,1})\right) = \psi_{\iota_{\CO_\CX},pq}(\Theta_{1,1}).\]
We now fix $(\nu_i)_{i \in I}$ with $I$ finite for $\theta_p$ as in Lemma~\ref{lem:loc homeo right reconstruction}. In the same way, we choose $(\mu_j)_{j \in J}$ for $\theta_q$. Then Lemma~\ref{lem:loc homeo right reconstruction} says that these two families satisfy the reconstruction formula for $p$ and $q$, respectively. Therefore, they fulfill
$\sum_{i \in I}\Theta_{\nu_i,\nu_i} = 1_{\CL(\CX_p)}$ and $\sum_{j \in J}\Theta_{\mu_j,\mu_j} = 1_{\CL(\CX_q)}$. Next, we compute
\[\begin{array}{lcl}
\psi_{\iota_{\CO_\CX},pq}\left(\iota_p^{pq}(\Theta_{1,1})\iota_q^{pq}(\Theta_{1,1})\right) &\hspace*{-3mm}=\hspace*{-3mm}& \sum\limits_{\substack{i \in I\\j \in J}}\psi_{\iota_{\CO_\CX},pq}\left(\Theta_{\alpha_p(\mu_j),\alpha_p(\mu_j)}\Theta_{\alpha_q(\nu_i),\alpha_q(\nu_i)}\right)\\
&\hspace*{-3mm}=\hspace*{-3mm}& \sum\limits_{\substack{i \in I\\j \in J}}\psi_{\iota_{\CO_\CX},pq}\left(\Theta_{\alpha_p(\mu_j\alpha_q(L_{pq}(\alpha_p(\mu_j)\alpha_q(\nu_i)))),\alpha_q(\nu_i)}\right)\\
&\hspace*{-3mm}=\hspace*{-3mm}& \sum\limits_{\substack{i \in I\\j \in J}}\psi_{\iota_{\CO_\CX},pq}\left(\Theta_{\alpha_p(\mu_jE_q(\mu_j L_p(\alpha_q(\nu_i)))),\alpha_q(\nu_i)}\right)\\
&\hspace*{-3mm}=\hspace*{-3mm}& \sum\limits_{i \in I}\psi_{\iota_{\CO_\CX},pq}\left(\Theta_{E_p(\alpha_q(\nu_i)),\alpha_q(\nu_i)}\right),\\
\end{array}\] 
where we used the (internal) reconstruction formula for $(\mu_j)_{j \in J}$ in the last step, compare Lemma~\ref{lem:loc homeo right reconstruction}. Since $p$ and $q$ are relatively prime, $\theta_p$ and $\theta_q$ $\ast$-commute by Definition~\ref{def:O[X,P,theta]}. So Proposition~\ref{prop:star-com gives comm endo and transfer op} implies that $E_p(\alpha_q(f)) = \alpha_q(E_p(f))$ holds for all $f \in C(X)$. Therefore, we have shown that  
\[\begin{array}{c} \psi_{\iota_{\CO_\CX},p}(\Theta_{1,1})\psi_{\iota_{\CO_\CX},q}(\Theta_{1,1}) = \psi_{\iota_{\CO_\CX},pq}\left(\sum\limits_{i \in I} \Theta_{\alpha_q(E_p(\nu_i)),\alpha_q(\nu_i)}\right). \end{array}\]
Applying $\sum\limits_{i \in I}\Theta_{\alpha_q(E_p(\nu_i)),\alpha_q(\nu_i)}$ to an element $f \in \CX_{pq}$ takes the form
\[\begin{array}{c} \sum\limits_{i \in I}\alpha_q(E_p(\nu_i))\alpha_q(E_p(\nu_i L_q(f))) = \sum\limits_{i \in I}\alpha_{pq}(L_p(\nu_i)L_p(\nu_iL_q(f))). \end{array}\]
From Lemma~\ref{lem:loc homeo right reconstruction} and $L_p(g E_p(h)) = L_p(g)L_p(h)$ for arbitrary $g,h \in C(X)$, see Definition~\ref{def:transfer operator}, we deduce  
\[\begin{array}{c}  \sum\limits_{i \in I}\Theta_{\alpha_q(E_p(\nu_i)),\alpha_q(\nu_i)}(f) = E_{pq}(f) = \Theta_{1,1}(f) \text{ in } \CX_{pq}. \end{array}\]
Since $f$ was arbitrary, we get 
\[\begin{array}{c} \sum\limits_{i \in I}\Theta_{\alpha_q(E_p(\nu_i)),\alpha_q(\nu_i)} = \Theta_{1,1} \text{ in } \CL(\CX_{pq}) \end{array}\]
and hence (III) holds. This shows that the map $\varphi$ is a $\ast$-homomorphism from $\CO[X,P,\theta]$ onto $\CO_{\CX}$. For the reverse direction, we show that 
\[\begin{array}{rccl}
\varphi_{\text{CNP},p}:& \CX_p &\longrightarrow& \CO[X,P,\theta]\\
& f &\mapsto& fs_p
\end{array}\]
defines a $(\text{CNP})$-covariant representation of $\CX$. Clearly, $\varphi_{\text{CNP}}$ satisfies (1) and (2) from Definition~\ref{def:rep of prod sys}. For (3), note that 
\[\varphi_{\text{CNP},p}(f)^*\varphi_{\text{CNP},p}(g) = s_p^*\overline{f}gs_p \stackrel{(III)}{=} L_p(\overline{f}g) = \varphi_{\text{CNP},1_P}(\langle f,g \rangle_p)\]
holds for all $p \in P$ and $f,g \in C(X)$. (4) follows from 
\[\varphi_{\text{CNP},p}(f)\varphi_{\text{CNP},q}(g) = fs_pgs_q \stackrel{(II)}{=} f\alpha_p(g)s_{pq} = \varphi_{\text{CNP},pq}(f\alpha_p(g)).\]
We only have to show $(CP_F)$-covariance in order to get that $\tilde{\varphi}$ is $(\text{CNP})$-covariant. To verify this, we fix $(\nu_i)_{I \in I} \subset C(X)$ with $I$ finite for $p \in P$ as in Lemma~\ref{lem:loc homeo right reconstruction} and obtain  
\[\begin{array}{lclcl}
\psi_{\varphi_{\text{CNP}},p}(\phi_p(f)) &=& \psi_{\varphi_{\text{CNP}},p}\left(\sum\limits_{i \in I}\Theta_{f\nu_i,\nu_i} \right) &=& f\sum\limits_{i \in I} \nu_is_ps_p^*\nu_i\\
&\stackrel{(IV)}{=}& f &=& \varphi_{\text{CNP},1_P}(f)
\end{array}\]
for all $f \in C(X)$. Thus, $\varphi_{\text{CNP}}$ is a $(\text{CNP})$-covariant representation of $\CX$. It is apparent that the induced $\ast$-homomorphism  $\overline{\varphi}_{\text{CNP}}:\CO_\CX \longrightarrow \CO[X,P,\theta]$ is the inverse of $\varphi$.
\end{proof}

\noindent It is conceivable that a similar result holds for the Nica-Toeplitz algebra $\CN\CT_\CX$, where relation $(IV)$ has to be weakened in the natural way.

\section{\texorpdfstring{Characterizing topological freeness of $(X,P,\theta)$ with $\CO[X,P,\theta]$}{Characterizing topological freeness}}
\noindent
In this section we establish an equivalence between topological freeness for irreversible $*$-commutative dynamical systems of finite type $(X,P,\theta)$ and three different C*-algebraic properties of $\CO[X,P,\theta]$, see Theorem~\ref{thm:top free char}. The proof of this result essentially relies on Proposition~\ref{prop:top free gives IIP}, where we prove that topological freeness of $(X,P,\theta)$ implies the ideal intersection property for $C(X)$ in $\CO[X,P,\theta]$. Moreover, we need the technical Lemma~\ref{lem:tech lemma top free}, which uses a faithful version $\tilde{\varphi}$ of the representation $\varphi$ from Proposition~\ref{prop:elementary rep SIDoFT}. In fact, Lemma~\ref{lem:tech lemma top free} is a straightforward generalization of \cite{CS}*{Lemma 5} to the setting of irreversible $*$-commutative dynamical systems of finite type.

Recall that $P$ is an Ore semigroup with enveloping group $P^{-1}P$ denoted by $H$. In the following, $(\xi_{x,h})_{(x,h) \in X{\times}H}$ denotes the standard orthonormal basis of $\ell^2(X{\times}H)$.

\begin{proposition}\label{prop:augm rep}
For $f \in C(X), (x,h) \in X\times H$, let $\tilde{M}_{f} \xi_{x,h} := f(x) \xi_{x,h}$ and $\tilde{S}_{p} \xi_{x,h} = N_{p}^{-\frac{1}{2}} \sum_{y \in \theta_{p}^{-1}(x)}{e_{y,ph}}$. Then $fs_p \mapsto \tilde{M}_f\tilde{S}_p$ defines a representation $\tilde{\varphi}$ of $\CO[X,P,\theta]$ on $\ell^{2}(X \times H)$, which is faithful on $C(X)$.
\end{proposition} 
\begin{proof}
As ${\tilde{S}_{p}}^{*} \xi_{x,h} = N_{p}^{-\frac{1}{2}} \xi_{\theta_{p}(x),p^{-1}h}$, the proof of Proposition~\ref{prop:elementary rep SIDoFT} carries over verbatim.
\end{proof}

\begin{remark}\label{rem:gauge-inv transfer}
As in \cite{CS}*{Proposition 4}, we would like to show that $\tilde{\varphi}$ is faithful by using a gauge-invariant uniqueness theorem. For this purpose, let us recall that Theorem~\ref{thm:isom ad-hoc PS for SIDoFT} asserts that $\CO[X,P,\theta]$ is the Cuntz-Nica-Pimsner algebra for the product system of Hilbert bimodules associated to $(X,P,\theta)$ in Proposition~\ref{prop:PS for an SIDoFT}. We intend to make use of \cite{CLSV}*{Corollary 4.12 (iv)} and remark that the terminology related to coactions can be phrased in terms of actions of the dual group of the discrete, abelian group $H = P^{-1}P$, which we denote by $L$. Under this transformation, the coaction $\delta$ in \cite{CLSV}*{Proposition 3.5} corresponds to the natural gauge action $\gamma$ of $L$ on $\CO[X,P,\theta]$ given by  $\gamma_\ell(f) = f \text{ and } \gamma_\ell(s_{p}) = \ell(p)s_{p}$ for $f \in C(X),p \in P$ and $\ell \in L$. Thus \cite{CLSV}*{Definition 4.10} and \cite{CLSV}*{Corollary 4.12 (iv)} imply that $\CO[X,P,\theta]$ has the following gauge-invariant uniqueness property: A surjective $\ast$-homomorphism $\phi:\CO[X,P,\theta] \longrightarrow B$ onto a C*-algebra $B$ is injective if and only if the following two conditions hold:
\begin{enumerate}
\item[a)] There is an $L$-action $\beta$ on $B$ for which $\phi$ is $(\gamma,\beta)$-equivariant. 
\item[b)] $\phi$ is faithful on $C(X)$. 
\end{enumerate} 
\end{remark}

\noindent This enables us to prove the analogue of \cite{CS}*{Proposition 4}:

\begin{proposition}\label{prop:augm rep is faithful}
The representation $\tilde{\varphi}: X \longrightarrow \CL(\ell^2(X \times H))$ is faithful.
\end{proposition}
\begin{proof}
Faithfulness of $\tilde{\varphi}$ on $C(X)$ has been established in Proposition~\ref{prop:augm rep}. For $\ell \in L$ define $U_\ell \in \CL\left(\ell^{2}(X \times H)\right)$ by $U_\ell \xi_{x,h} = \ell(h) \xi_{x,h}$. This gives a unitary representation of $L$ and enables us to define an action $\beta$ of $L$ on $\CL\left(\ell^{2}(X \times H)\right)$ via $\beta_\ell\left(T\right) = U_\ell T U_\ell^*$. We observe that, on $\tilde{\varphi}\left(\CO[X,P,\theta]\right)$, $\beta$ is given by $\beta_\ell(\tilde{M}_{f}) = \tilde{M}_{f} \text{ and } \beta_\ell(\tilde{S}_{p}) = \ell(p) \tilde{S}_{p}$ for all $f \in C(X)$ and $p \in P$. Thus $\tilde{\varphi}$ is $(\gamma,\beta)$-equivariant. According to the conclusion of Remark~\ref{rem:gauge-inv transfer}, $\tilde{\varphi}$ is faithful on all of $\CO[X,P,\theta]$.
\end{proof}

\noindent Recall from Corollary~\ref{cor:cond exp to C(X)}, that $G:\CO[X,P,\theta] \longrightarrow C(X)$ is the conditional expectation given by $G(fs_ps_q^*g) = \delta_{p q}~N_p^{-1} fg$.

\begin{lemma}\label{lem:tech lemma top free}
Let $\tilde{\varphi}$ be the representation from Proposition~\ref{prop:augm rep} and $a \in \CO[X,P,\theta]$. Then the following statements hold:
\begin{enumerate}[i)]
\item $\left\langle \tilde{\varphi}(a) \xi_{x,h}, \xi_{x,h} \right\rangle = G(a)(x)$ for all $(x,h) \in X \times H$. 
\item $\left\langle \tilde{\varphi}(a) \xi_{x_1,h_1}, \xi_{x_2,h_2} \right\rangle = 0$ holds for all $(x_1,h_1) \neq (x_2,h_2)$ if and only if $a \in C(X)$.
\item If $(x_1,h_1),(x_2,h_2) \in X \times H$ satisfy $\left\langle \tilde{\varphi}(a) \xi_{x_1,h_1}, \xi_{x_2,h_2} \right\rangle \neq 0$,	there are $p,q \in P$ and open neighbourhoods $U_{1}$ of $x_{1}$, $U_{2}$ of $x_{2}$ with the following properties:
\begin{enumerate}[(a)]
\item $ph_1 = qh_2$. 
\item $\theta_q(x_1) = \theta_p(x_2)$.
\item Whenever $x_{3} \in U_{1}$ and $x_{4} \in U_{2}$ satisfy $\theta_q(x_{3}) = \theta_p(x_{4})$, then $\left\langle \tilde{\varphi}(a) \xi_{x_3,h_1}, \xi_{x_4,h_2} \right\rangle \neq 0$.
\end{enumerate}	
\end{enumerate}
\end{lemma}
\begin{proof}
Recall that the linear span of $\{ fs_ps_q^*g \mid f,g \in C(X),~p,q \in P\}$ is dense in $\CO[X,P,\theta]$ according to Lemma~\ref{lem:O-alg SIDoFT - lin span}. As both sides of the equation in i) are linear and continuous in $a$, it suffices to prove the equation for $a = fs_ps_q^*g$. This is achieved by 
\[\begin{array}{lcl}
\left\langle \xi_{x,h}, \tilde{\varphi}(fs_ps_q^{*}g)\xi_{x,h} \right\rangle &=& \left\langle \tilde{\varphi}(s_p^{*}\overline{f})\xi_{x,h}, \tilde{\varphi}(s_q^{*}g)\xi_{x,h} \right\rangle \vspace{2mm}\\
&=& \delta_{p q}~N_p^{-1}~f(x)g(x) \vspace{2mm}\\
&=& G(fs_ps_q^{*}g)(x)
\end{array}\]

For ii), we note that $a \in C(X)$ certainly implies $\left\langle \tilde{\varphi}(a) \xi_{x_1,h_1}, \xi_{x_2,h_2} \right\rangle = 0$ for all $(x_1,h_1) \neq (x_2,h_2)$. On the other hand, if $a \in \CO[X,P,\theta]$ satisfies $\left\langle \tilde{\varphi}(a) \xi_{x_{1},h_{1}}, \xi_{x_{2},h_{2}} \right\rangle = 0$ whenever $(x_1,h_1) \neq (x_2,h_2)$, part i) implies $\tilde{\varphi}(a) = \tilde{\varphi}(G(a))$. As $\tilde{\varphi}$ is faithful, see Proposition~\ref{prop:augm rep is faithful}, this shows $a = G(a) \in C(X)$.

In order to prove iii), suppose we have $(x_1,h_1),(x_2,h_2) \in X \times H$ such that $\varepsilon := | \left\langle \tilde{\varphi}(a) \xi_{x_{1},h_{1}}, \xi_{x_{2},h_{2}} \right\rangle | > 0.$ According to Lemma~\ref{lem:O-alg SIDoFT - lin span}, we can choose $p_1,q_1,\dots,p_m,q_m \in P$ and $f_1,g_1,\dots,f_m,g_m \in C(X)$ such that 
\[\begin{array}{c} a_{m} :=  \sum\limits_{i = 1}^m{f_is_{p_i}s_{q_i}^{*}g_i} \text{ satisfies } \|a - a_m \| < \frac{\varepsilon}{3}. \end{array}\]
As \label{basic comp}
\[\begin{array}{c} \tilde{\varphi}\left(f_is_{p_i}s_{q_i}^{*}g_i\right) \xi_{x_1,h_1} = N_{p_iq_i}^{-\frac{1}{2}} \sum\limits_{y \in \theta_{p_i}^{-1}\left(\theta_{q_i}(x_1)\right)}{\hspace*{-2mm} f_i(y)g_i(x_1) \xi_{y,p_iq_i^{-1}h_1}}, \end{array}\]
we either get $\left\langle \tilde{\varphi}\left(f_is_{p_i}s_{q_i}^{*}g_i\right) \xi_{x_1,h_1}, \xi_{x_2,h_2} \right\rangle = 0$ or $x_2 \in \theta_{p_i}^{-1}\left(\theta_{q_i}(x_1)\right)$ and $p_iq_i^{-1}h_1 = h_2$. The latter conditions are equivalent to $\theta_{p_i}(x_2) = \theta_{q_i}(x_1)$ and $p_ih_1 = q_ih_2$ since $P$ is commutative. Note that there is at least one $i$ such that 
\[\left\langle \tilde{\varphi}\left(f_is_{p_i}s_{q_i}^{*}g_i\right) \xi_{x_1,h_1}, \xi_{x_2,h_2} \right\rangle \neq 0\]
because $\tilde{\varphi}$ is contractive and $\|a - a_{m} \| < \frac{\varepsilon}{3}$. Therefore, possibly changing the enumeration, we can assume that there is $1 \leq n \leq m$ such that 
\[\left\langle \tilde{\varphi}\left(f_is_{p_i}s_{q_i}^{*}g_i\right) \xi_{x_1,h_1}, \xi_{x_2,h_2} \right\rangle \neq 0 \text{ if and only if } 1 \leq i \leq n.\] 
Let $a_{n} := \sum_{i = 1}^{n}{f_{i,1}s_{p_{i,1}}s_{p_{i,2}}^{*}f_{i,2}}$. Since $P$ is lattice ordered there is a unique element $p_0 := p_1 \vee \dots \vee p_n$. Additionally, set $q_0 := h_2^{-1}ph_1 \in H$ and note that $q_0 \in P$ since $h_2^{-1}p_ih_1 = q_i \in P$ for all $1 \leq i \leq n$. For each $1 \leq i \leq n$, there are open neighbourhoods $U_{i,1}^{'}$ of $x_1$ and $U_{i,2}^{'}$ of $x_2$ such that 
\begin{enumerate}
\item[$\bullet$] $\theta_{p_i}$ is injective on $U_{i,1}^{'}$,
\item[$\bullet$] $\theta_{q_i}$ is injective on $U_{i,2}^{'}$, and 
\item[$\bullet$] $N_{p_iq_i}^{-\frac{1}{2}} ~ | f_i(y_1)g_i(y_2) - f_i(x_1)g_i(x_2) | < \frac{\varepsilon}{3n} \text{ for all } y_1 \in U_{i,1}^{'},y_2 \in U_{i,2}^{'}.$  
\end{enumerate}

\noindent This is always possible because the transformations $\theta_{p_i},\theta_{q_i}$ are local homeomorphisms and the function $X^2 \longrightarrow \IC$ given by $(y_{1},y_{2}) \mapsto f_{i,1}(y_{1})f_{i,2}(y_{2})$ is continuous. Then 
\[U_{i,1} := \theta_{p_i}^{-1}\left(\theta_{p_i}(U_{i,1}^{'}) \cap \theta_{q_i}(U_{i,2}^{'})\right)\]
defines an open neighbourhood of $x_1$ such that for each $y_1 \in U_{i,1}$ there is a unique $y_2 \in U_{i,2}$ with $\theta_{p_i}(y_2) = \theta_{q_i}(y_1)$. Accordingly, set 
\[U_{i,2} := \theta_{q_i}^{-1}\left(\theta_{p_i}(U_{i,1}^{'}) \cap \theta_{q_i}(U_{i,2}^{'})\right).\]
and take $U_{j} := \bigcap_{i = 1}^{n}{U_{i,j}}$ for $j = 1,2$. Now suppose $x_{3} \in U_{1},~ x_{4} \in U_{2}$ satisfy $\theta_q(x_{3}) = \theta_p(x_{4})$. Using the triangle inequality for the first two steps, we get
\[\begin{array}{lclcl}
| \left\langle \tilde{\varphi}(a) \xi_{x_{3},h_1}, \xi_{x_{4},h_{2}} \right\rangle | 
&\hspace{-2mm}\geq\hspace{-2mm}& \varepsilon &\hspace{-2mm}-\hspace{-2mm}&\hspace{-2mm} | \left\langle \tilde{\varphi}(a) \xi_{x_{1},h_1}, \xi_{x_{2},h_{2}} \right\rangle - \left\langle \tilde{\varphi}(a) \xi_{x_{3},h_1}, \xi_{x_{4},h_{2}} \right\rangle | \vspace{2mm}\\
&\hspace{-2mm}\geq\hspace{-2mm}& \varepsilon &\hspace{-2mm}-\hspace{-2mm}&\hspace{-2mm} | \left\langle \tilde{\varphi}(a) \xi_{x_{1},h_1}, \xi_{x_{2},h_{2}} \right\rangle - \left\langle \tilde{\varphi}(a_{m}) \xi_{x_{1},h_1}, \xi_{x_{2},h_{2}} \right\rangle | \vspace{2mm}\\
&\hspace{-4mm}&&\hspace{-2mm}-\hspace{-2mm}&\hspace{-2mm} | \left\langle \tilde{\varphi}(a_{m}) \xi_{x_{1},h_1}, \xi_{x_{2},h_{2}} \right\rangle - \left\langle \tilde{\varphi}(a_{m}) \xi_{x_{3},h_1}, \xi_{x_{4},h_{2}} \right\rangle | \vspace{2mm}\\
&\hspace{-4mm}&&\hspace{-2mm}-\hspace{-2mm}&\hspace{-2mm} | \left\langle \tilde{\varphi}(a_{m}) \xi_{x_{3},h_1}, \xi_{x_{4},h_{2}} \right\rangle - \left\langle \tilde{\varphi}(a) \xi_{x_{3},h_1}, \xi_{x_{4},h_{2}} \right\rangle | \vspace{2mm}\\	
&\hspace{-2mm}=\hspace{-2mm}& \varepsilon &\hspace{-2mm}-\hspace{-2mm}&\hspace{-2mm} | \left\langle \tilde{\varphi}(a-a_{m}) \xi_{x_{1},h_1}, \xi_{x_{2},h_{2}} \right\rangle | \vspace{2mm}\\
&\hspace{-4mm}&&\hspace{-2mm}-\hspace{-2mm}&\hspace{-2mm} | \left\langle \tilde{\varphi}(a_{n}) \xi_{x_{1},h_1}, \xi_{x_{2},h_{2}} \right\rangle - \left\langle \tilde{\varphi}(a_{n}) \xi_{x_{3},h_1}, \xi_{x_{4},h_{2}} \right\rangle | \vspace{2mm}\\
&\hspace{-4mm}&&\hspace{-2mm}-\hspace{-2mm}&\hspace{-2mm} | \left\langle \tilde{\varphi}(a_{m}-a) \xi_{x_{3},h_1}, \xi_{x_{4},h_{2}} \right\rangle | \vspace{2mm}\\									
&\hspace{-2mm}>\hspace{-2mm}& \varepsilon &\hspace{-2mm}-\hspace{-2mm}&\hspace{-2mm} \frac{\varepsilon}{3} - n\frac{\varepsilon}{3n} - \frac{\varepsilon}{3} = 0.
\end{array}\]	
\end{proof}

\noindent This marks the end of the first half of the preparations for Theorem~\ref{thm:top free char}. The second part will show that topological freeness of $(X,P,\theta)$ results in the ideal intersection property for $C(X)$ inside $\CO[X,P,\theta]$, see Proposition~\ref{prop:top free gives IIP}. 

\begin{lemma}\label{lem:aux lemma top free gives IIP}
Let $p,q \in P$. If $x \in X$ satisfies $\theta_p(x) \neq \theta_q(x)$, then there exists a positive contraction $h \in C(X)$ such that $h(x) = 1$ and $hs_ps_q^{*}h = 0$.
\end{lemma}
\begin{proof} 
The steps leading to a proof are:
\begin{enumerate}[a)]
\item There is an open neighbourhood $U$ of $x$ satisfying $U \cap \theta_q^{-1}(\theta_p(U)) = \emptyset$.
\item $\supp L_{r}(f) \subset \theta_r(\supp f)$ holds for all $f \in C(X)$ and $r \in P$.
\item There exists a positive contraction $h \in C(X)$ with $h(x) =1$ and $\supp h \subset U$ for the $U$ obtained in a). Every $h$ of this form satisfies $h\alpha_q(L_p(h^2)) = 0$.
\end{enumerate}
As $X$ is Hausdorff, there are disjoint, open neighbourhoods $V$ and $W$ of $\theta_p(x)$ and $\theta_q(x)$, respectively. Hence $U := \theta_p^{-1}(V) \cap \theta_q^{-1}(W)$ is an open neighbourhood of $x$ and $\theta_p(U) \cap \theta_q(U) \subset V \cap W = \emptyset$, so 
\[U \cap \theta_q^{-1}(\theta_p(U)) \subset \theta_q^{-1}(\theta_q(U)) \cap \theta_q^{-1}(\theta_p(U)) = \theta_q^{-1}(\theta_p(U) \cap \theta_q(U)) = \emptyset\]
establishes a). The proof of claim b) is straightforward. For the first part of c), we note that such an $h$ exists because $U$ is an open neighbourhood of $x$ and $X$ is a normal space. Therefore we get 
\[\supp h\alpha_q(L_p(h^2)) \subset U \cap \theta_q^{-1}(\supp L_p(h^2)) 
\stackrel{b)}{\subset} U \cap \theta_q^{-1}(\theta_p(\underbrace{\supp h^2}_{\subset U})) \stackrel{a)}{=} \emptyset\]
which proves $h\alpha_q(L_p(h^2)) = 0$. Combining these ingredients, we deduce
\[\|hs_ps_q^{*}h\|^{2} = \|hs_qs_p^{*}h^{2}s_ps_q^{*}h\| 
= \|h~\alpha_q(L_p(h^2))s_qs_q^{*}h\| \stackrel{c)}{=} 0.\]
\end{proof}

\begin{remark}
Observe that we can deduce from the proof of Lemma~\ref{lem:aux lemma top free gives IIP} that condition ii) is equivalent to $h\alpha_{p_1}(L_{p_2}(h)) = 0$ as well as to $h\alpha_{p_2}(L_{p_1}(h)) = 0$. 
\end{remark}

\noindent Before we reach the central result of this section, let us recall the notion of topological freeness for dynamical systems, where the transformations need not be reversible.

\begin{definition}\label{def:top freeness}
A topological dynamical system consisting of a topological space $Y$ and a semigroup $S$ together with an action $S \stackrel{\eta}{\curvearrowright}Y$ by continuous transformations is said to be \textit{topologically free} if the set $\{y \in Y \mid \eta_s(y) = \eta_t(y) \}$ has empty interior for all $s,t \in S,~s \neq t$.
\end{definition}

\begin{proposition}\label{prop:top free gives IIP}
If $(X,P,\theta)$ is topologically free, every non-zero ideal $I$ in $\CO[X,P,\theta]$ satisfies $I \cap C(X) \neq 0$. 
\end{proposition}
\begin{proof}
We will follow the strategy from \cite{EV}*{Theorem 10.3}. Suppose $I$ is an ideal in $\CO[X,P,\theta]$ satisfying $I \cap C(X) = 0$ and denote by $\pi$ the corresponding quotient map. Then $\pi$ is isometric on $C(X)$. We claim that $\|\pi(a)\| \geq \|G(a)\| \text{ holds for all positive } a \in \CO[X,P,\theta]$. By continuity (of the norms, of $\pi$, and of $G$), it suffices to prove the above equation for 
\[\begin{array}{c} a = \sum\limits_{j = 1}^{n}{f_js_{p_j}s_{q_j}g_j}, \text{ with } n \in \N, f_j,g_j \in C(X) \text{ and } p_j,q_j \in P. \end{array}\]
Without loss of generality, we can assume that there is $1 \leq n_0 \leq n$ such that $p_j = q_j$ holds if and only if $j \leq n_0$. In fact, possibly inflating the elements $f_js_{p_j}s_{q_j}^{*}g_j$ by $1 = \sum\limits_{i}{\nu_is_ps_p^{*}\nu_i}$ for $p \geq p_1 \vee \dots \vee p_{n_0}$, see Lemma~\ref{lem:CNP equiv}, we can assume that $p_j = q_j = p$ holds for all $1 \leq j \leq n_{0}$. 

If $n_{0} = n$, then we have $a \in (\CF_{p})_{+}$. In this case, set $h = 1$. In case $n_{0} < n$, $\bigcap_{j = n_{0}+1}^{n}{\{x \in X \mid \theta_{p_j}(x) \neq \theta_{q_j}(x)\}}$ is dense in $X$ since $(X,P,\theta)$ is topologically free. Thus, for each $\varepsilon \in (0,1)$, there exists $x \in X$ satisfying
\begin{enumerate}[a)]
\item $G(a)(x) > (1-\varepsilon)~\|G(a)\|$, and 
\item $\theta_{p_j}(x) \neq \theta_{q_j}(x)$ for all $n_{0} < j \leq n$.
\end{enumerate} 
Applying Lemma~\ref{lem:aux lemma top free gives IIP} to each $n_{0} < j \leq n$ yields functions $h_{n_0+1},\dots,h_n \in C(X)$, which we use to build $h := \prod_{j = n_{0}+1}^{n}{h_{j}}$. Then $h$ satisfies
\begin{enumerate}[(a)]
\item $0 \leq h \leq 1$, 
\item $h(x) = 1$, and 
\item $hs_{p_j}s_{q_j}^{*}h = 0$ for all $n_{0} < j \leq n$.
\end{enumerate}
This results in 
\[\begin{array}{c} hah = \sum\limits_{j = 1}^{n}{f_jhs_{p_j}s_{q_j}^{*}hg_j} \stackrel{(c)}{=} \sum\limits_{j = 1}^{n_{0}}{f_jhs_{p}s_{p}^{*}hg_j} = hE_1(a)h, \end{array}\]
where $E_1:\CO[X,P,\theta] \longrightarrow \CF$ is the faithful conditional expectation from Lemma~\ref{lem:cond exp to F}. Note that we have $E_1(a) = \sum_{j=1}^{n_0} f_js_ps_p^*g_j \in (\CF_{p})_{+}$. Next, choose a partition of unity $(v_k)_{1 \leq k \leq m}$ for $X$ and $\theta_p$ as in Lemma~\ref{lem:loc homeo right reconstruction} and, as before, let $\nu_k := (N_pv_k)^{\frac{1}{2}}$. Then we obtain
\[\begin{array}{lclcl}
G(a) &=& G(E_{1}(a)) &=& N_{p}^{-1}~\sum\limits_{j = 1}^{n_{0}}{f_jg_j} \vspace*{2mm}\\
&&&=& N_{p}^{-1}~\sum\limits_{j = 1}^{n_{0}}{f_j\left(\sum\limits_{k = 1}^{m}{\nu_{k}s_{p}s_{p}^{*}\nu_{k}}\right)g_j}\vspace*{2mm}\\ 
&&&=& \sum\limits_{k = 1}^{m}{v_{k}^{\frac{1}{2}}E_{1}(a)v_{k}^{\frac{1}{2}}}.
\end{array}\]
Combining this with the fact that $\pi(a) \mapsto \sum\limits_{k = 1}^{m}{\pi(v_{k}^{\frac{1}{2}})~\pi(a)~\pi(v_{k}^{\frac{1}{2}})}$ is a unital completely positive map, hence contractive, we get
\[\begin{array}{lcl}
\|\pi(a)\| &\geq& \| \pi\left(\sum\limits_{k = 1}^{m}{v_{k}^{\frac{1}{2}}av_{k}^{\frac{1}{2}}}\right)\| \vspace*{2mm}\\
&\geq& \| \pi\left(\sum\limits_{k = 1}^{m}{v_{k}^{\frac{1}{2}}hahv_{k}^{\frac{1}{2}}}\right)\| \vspace*{2mm}\\
&=& \| \pi\left(h\sum\limits_{k = 1}^{m}{v_{k}^{\frac{1}{2}}E_{1}(a)v_{k}^{\frac{1}{2}}}h\right)\| \vspace*{2mm}\\
&=& \|\pi(hG(a)h)\|\\
&=& \|hG(a)h\|
\end{array}\]
since $\pi$ is isometric on $C(X)$. On the other hand, 
\[\|hG(a)h\| \geq (hG(a)h)(x) = G(a)(x) > (1-\varepsilon)\|G(a)\|,\]
so $\|\pi(a)\| > (1-\varepsilon)\|G(a)\|$ for all $\varepsilon > 0$. This forces $\|\pi(a)\| \geq \|G(a)\|$. 

So given $a \in \CO[X,P,\theta]_{+} \cap I$, we have $0 = \|\pi(bab^{*})\| \geq \|G(bab^{*})\|$ for all $b \in \CO[X,P,\theta]$. In particular, $G(bab^{*}) = 0$ holds for all $b \in \CF$. But according to Lemma~\ref{lem:weak faithfulness}, this implies $a = 0$ and hence $I = 0$.    
\end{proof}

\noindent We now state for the main result of this section:
\TopFreeChar
\begin{proof}
The plan is as follows:
\[\xymatrix{\text{(1)}\ar@{=>}[r]\ar@{<=>}[d] & \text{(2)}\ar@{=>}[d] \\
\text{(4)} & \text{(3)}\ar@{=>}[ul]	}\]
The implication from (1) to (2) is precisely covered by Proposition~\ref{prop:top free gives IIP} and (2) gives (3) because we have $\ker\varphi \cap C(X) = 0$, see Proposition~\ref{prop:elementary rep SIDoFT}. Next, we show that (3) or (4) implies (1), where we proceed by contraposition. If the system is not topologically free, there are $p,q \in P$ with $p \neq q$ such that $\left\{x \in X \mid \theta_p(x) = \theta_q(x)\right\}$ has non-empty interior. Since the maps $\theta_p$ and $\theta_q$ are local homeomorphisms, there exists a non-empty open $U \subset \left\{x \in X \mid \theta_p(x) = \theta_q(x)\right\}$ such that $\theta_p|_{U} = \theta_q|_{U}$ is injective. We fix $x_{0} \in U$ and choose a positive $f \in C(X)$ satisfying $f(x_{0}) \neq 0$ and $\supp f \subset U$. By appealing to the existence of partitions of unity for open covers of compact Hausdorff spaces, we know that such a function $f$ always exists. Let us point out that $fs_ps_q^{*}f$ does not belong to $C(X)$, which can formally be deduced from Lemma~\ref{lem:tech lemma top free}~ii), $p \neq q$, and 
\[\left\langle \tilde{\varphi}(fs_ps_q^{*}f) \xi_{x_{0},q}, \xi_{x_{0},p} \right\rangle = N_{pq}^{-\frac{1}{2}} f(x_{0})^{2} \neq 0.\] 
Then
\[\begin{array}{lcl}
\left\langle \varphi(fs_ps_q^{*}f) \xi_x, \xi_y \right\rangle &=& \left\langle \varphi(s_q^{*}f) \xi_x, \varphi(s_p^{*}f)\xi_y \right\rangle \vspace*{2mm}\\
&=& \delta_{\theta_q(x) \theta_p(y)} ~ N_{pq}^{-\frac{1}{2}} ~ f(x)f(y)\vspace*{2mm}\\
&=& \delta_{x y} ~ N_{pq}^{-\frac{1}{2}} ~ f(x)^{2}
\end{array}\]
holds for all $x,y \in U$, where we used injectivity of $\theta_p|_{U} = \theta_q|_{U}$. Note that the expression vanishes whenever $x$ or $y$ is not contained in $U$ due to $\supp f \subset U$. Hence we get $0 \neq fs_ps_q^{*}f - N_{pq}^{-\frac{1}{2}} f^{2} \in \ker\varphi$, which shows that (3) implies (1).

In order to prove that (4) forces (1), it suffices to show that the function $f$ from the last part satisfies $fs_ps_q^{*}f \in C(X)^{'} \cap \CO[X,P,\theta]$. Let us pick $(\nu_{i})_{i \in I}$ for $\theta_q$ as in Lemma~\ref{lem:loc homeo right reconstruction}. We claim that
\[\alpha_p(L_q(fg\nu_{i})) = N_q^{-1} gf\nu_{i} = g\alpha_p(L_q(f\nu_{i}))\]
holds for all $g \in C(X)$ and $i \in I$. Using the property that $\theta_q|_{\supp \nu_i}$ is injective, it is easy to see that the functions match on $X \setminus \supp f$, so let $x \in \supp f \subset U.$ Then
\[\begin{array}{lcl}
\alpha_p(L_q(fg\nu_{i}))(x) &=& N_q^{-1} \sum\limits_{y \in \theta_q^{-1} (\theta_p(x))} g(y)f(y)\nu_{i}(y) \vspace{2mm}\\
&=& N_q^{-1} \sum\limits_{y \in \theta_q^{-1}(\theta_q(x))} g(y)f(y)\nu_{i}(y) \vspace{2mm}\\
&=& N_q^{-1} g(x)f(x)\nu_{i}(x)
\end{array}\]
holds, where we used $\theta_p|_{U} = \theta_q|_{U}$ and injectivity of $\theta_q|_{U}$. Similarly we get 
\[\begin{array}{lcl}
g\alpha_p(L_q(f\nu_{i}))(x) &=& g(x) N_q^{-1} \sum\limits_{y \in \theta_q^{-1}(\theta_p(x))} f(y)\nu_{i}(y) \vspace{2mm}\\
&=& N_q^{-1} g(x)f(x)\nu_{i}(x).
\end{array}\]
Thus $\alpha_p(L_q(fg\nu_{i})) = N_q^{-1} gf\nu_{i} = g\alpha_p(L_q(f\nu_{i}))$
is valid for all $g \in C(X)$ and $i \in I$. Using this equation, we deduce
\[\begin{array}{lcl}
fs_ps_q^{*}fg &=& fs_ps_q^{*}fg \sum\limits_{i \in I}{\nu_{i}s_qs_q^{*}\nu_{i}} \vspace{2mm}\\
&=& \sum\limits_{i \in I} f\alpha_p(L_q(fg\nu_{i})) s_ps_q^*\nu_i \vspace{2mm}\\
&=& gfs_p \sum\limits_{i \in I} L_q(f\nu_i) s_q^*\nu_i \vspace{2mm}\\
&=& gfs_ps_q^* \sum\limits_{i \in I} \nu_i E_q(\nu_i f) \vspace{2mm}\\
&=& gfs_ps_q^*f.
\end{array}\]
for arbitrary $g \in C(X)$. Thus, $fs_ps_q^*f \in \left(C(X)^{'} \cap \CO[X,P,\theta]\right) \setminus C(X)$, so $C(X)$ is not a masa in $\CO[X,P,\theta]$.

To deduce (4) from (1), let $a \in C(X)^{'} \cap \CO[X,P,\theta]$. By Lemma~\ref{lem:tech lemma top free}~ii), $a \in C(X)$ follows provided that $\left\langle \tilde{\varphi}(a) \xi_{x_{1},h_{1}}, \xi_{x_{2},h_{2}} \right\rangle = 0$ holds for all $(x_{1},h_{1}) \neq (x_{2},h_{2})$. In case $x_{1} \neq x_{2}$, there is $f \in C(X)$ satisfying $f(x_{1}) \neq 0$ and $f(x_{2}) = 0$. Thus
\[\begin{array}{lcl}
f(x_{1})\left\langle \tilde{\varphi}(a) \xi_{x_{1},h_{1}}, \xi_{x_{1},h_{2}} \right\rangle &=& \left\langle \tilde{\varphi}(af) \xi_{x_{1},h_{1}}, \xi_{x_{2},h_{2}} \right\rangle \vspace{2mm}\\
&=& \left\langle \tilde{\varphi}(fa) \xi_{x_{1},h_{1}}, \xi_{x_{2},h_{2}} \right\rangle \vspace{2mm}\\
&=& f(x_{2})\left\langle \tilde{\varphi}(a) \xi_{x_{1},h_{1}}, \xi_{x_{2},h_{2}} \right\rangle \vspace{2mm}\\
&=& 0
\end{array}\]
implies that $\left\langle \tilde{\varphi}(a) \xi_{x_{1},h_{1}}, \xi_{x_{1},h_{2}} \right\rangle = 0$. Now let $x_{1} = x_{2}$ and $h_{1} \neq h_{2}$ and we assume $\left\langle \tilde{\varphi}(a) \xi_{x_{1},h_{1}}, \xi_{x_{1},h_{2}} \right\rangle \neq 0$ in order to derive a contradiction: Part iii) from Lemma~\ref{lem:tech lemma top free} states that there are $p,q \in P$ and open neighbourhoods $U_{1},U_{2}$ of $x_{1} = x_{2}$ with the properties (a)-(c). Note that $p \neq q$ due to (a) and $h_{1} \neq h_{2}$. By passing to smaller neighbourhoods of $x_{1}$, if necessary, we may assume that for each $x_{3} \in U_{1}$ there is a unique $x_{4} \in U_{2}$ satisfying $\theta_q(x_{3}) = \theta_p(x_{4})$ (and vice versa). In other words, the (a priori multivalued) maps $\theta_q^{-1}\theta_p: U_1 \longrightarrow U_2$ and $\theta_p^{-1}\theta_q: U_2 \longrightarrow U_1$ are homeomorphisms. This uses the standing assumption that $\theta_p$ and $\theta_q$ are local homeomorphisms. As $(X,P,\theta)$ is topologically free, the set $\left\{x \in U_{1} \mid \theta_p(x) = \theta_q(x)\right\}$ has empty interior, so it cannot be all of $U_{1}$. Hence there are $x_{3} \in U_{1}$ and $x_{4} \in U_{2}$ such that $x_{3} \neq x_{4}$ and $\theta_q(x_{3}) = \theta_p(x_{4})$. Now Lemma~\ref{lem:tech lemma top free}~iii) implies $\left\langle \tilde{\varphi}(a) \xi_{x_{3},h_{1}}, \xi_{x_{4},h_{2}} \right\rangle \neq 0$. On the other hand, we observe that $\left\langle \tilde{\varphi}(a) \xi_{x_{3},h_{1}}, \xi_{x_{4},h_{2}} \right\rangle = 0$ follows from the consideration of the case $x_1 \neq x_2$ from before because $x_3 \neq x_4$. This reveals a contradiction and thus, $\left\langle \tilde{\varphi}(a) \xi_{x_{1},h_{1}}, \xi_{x_{2},h_{2}} \right\rangle = 0$ whenever $(x_{1},h_{1}) \neq (x_{2},h_{2})$. According to Lemma~\ref{lem:tech lemma top free}~ii), this forces $a \in C(X)$, so $C(X)$ is a masa in $\CO[X,P,\theta]$.
\end{proof}

\begin{remark}
The representation $\varphi$ is an analogue of the reduced representation for ordinary group crossed products, for if $\theta_{p}$ was a homeomorphism of $X$, then $S_{p}\xi_{x} = \xi_{\theta_{p}^{-1}(x)}$, see Proposition~\ref{prop:elementary rep SIDoFT}. Therefore condition (3) of Theorem~\ref{thm:top free char} can be interpreted as an amenability property of the dynamical system $(X,P,\theta)$, compare \cite{BO}*{Theorem 4.3.4}. Interestingly, this property coincides with topological freeness for irreversible $*$-commutative dynamical systems of finite type.
\end{remark}

\section{\texorpdfstring{Simplicity of $\CO[X,P,\theta]$}{Simplicity of the C*-algebra}}
\label{sec5}
\noindent Let $X$ be a compact Hausdorff space, $G$ a discrete group, and $\alpha$ an action of $G$ on $C(X)$. Then simplicity of the crossed product $C(X) \rtimes_{\alpha} G$ corresponds to minimality and topological freeness of the underlying topological dynamical system, given that the action $\alpha$ is amenable, see \cite{AS}*{Corollary following Theorem 2} or \cite{BO}*{Theorem 4.3.4 (1)}. An intermediate step for this result is to prove that every non-zero ideal $I$ in the C*-algebra $C(X) \rtimes_{\alpha} G$ satisfies $I \cap C(X) \neq 0$ if the dynamical system is topologically free, see \cite{AS}*{Theorem 2}. In view of Proposition~\ref{prop:top free gives IIP}, the analogous statement for $\CO[X,P,\theta]$ and $(X,P,\theta)$ has already been established. In fact, Theorem~\ref{thm:top free char} revealed that these conditions are equivalent.

But in contrast to the case of group actions, topological freeness is implied by minimality for irreversible $*$-commutative dynamical systems of finite type, see Proposition~\ref{prop:min -> top free}. The proof of this result is an adaptation of \cite{EV}*{Proposition 11.1}. Once this is accomplished, we show that $\CO[X,P,\theta]$ is simple if and only if $(X,P,\theta)$ is minimal, see Theorem~\ref{thm:simple iff minimal}. Hence we achieve a direct generalization of \cite{EV}*{Theorem 11.2}, if we suppress the additional requirement that each $\theta_p$ is assumed to be regular, see Definition~\ref{def:SID}. We note that this extra condition is assumed in \cite{EV}*{Section 9} as well, but not in \cite{EV}*{Sections 8,10 and 11}. As an application of Theorem~\ref{thm:simple iff minimal}, we characterize simplicity of $\CO[G,P,\theta]$ for commutative irreversible algebraic dynamical systems of finite type by minimality of $(G,P,\theta)$, see Corollary~\ref{cor:char min-simple for CIADoFT}.


\begin{definition}\label{def:orbit sgp action}
Let $Z$ be a topological space, $S$ a commutative semigroup and $S \stackrel{\eta}{\curvearrowright} Z$ a semigroup action by continuous maps. For $x \in Z$,
\[\CO(x) := \{\eta_t^{-1}(\eta_s(x)) \mid s,t \in S\} \subset Z\]
is called the \textit{orbit} of $x$ under $\eta$. Two elements $x,y \in Z$ are called \textit{orbit-equivalent}, denoted by $x \sim y$, if $\CO(x) = \CO(y)$.
\end{definition}

\noindent Note that $x$ and $y$ are orbit-equivalent if and only if there are $s,t \in S$ such that $\eta_s(x) = \eta_t(y)$. This definition is the natural generalization of trajectory-equivalence as defined in \cite{EV}*{Section 11}. $\sim$ is an equivalence relation because $S$ is commutative. 

\begin{definition}
$Y \subset Z$ is called \textit{invariant}, if $\eta_s^{-1}(Y) = Y$ for all $s \in S$.
\end{definition}

\noindent The proof of the following lemma is straightforward and therefore omitted.

\begin{lemma}\label{lem:inv preimage}
Let $Z$ be a topological space, $S$ a commutative semigroup and $S \stackrel{\eta}{\curvearrowright} Z$ a semigroup action by continuous, surjective maps. Then $Y \subset Z$ is invariant if and only if $x \sim y \in Y$ implies $x \in Y$ for all $x \in Z$.
\end{lemma}

\noindent In the case of actions by homeomorphisms, it is well-known that invariance of a subset passes to its closure. This is not clear for general irreversible transformations, but it is true for actions by local homeomorphisms.  
This is certainly well-known, but not easy to find in the literature, so we include a proof for convenience.

\begin{lemma}\label{lem:preimage and closure}
Let $Z$ be a topological space, $S$ a commutative semigroup and $S \stackrel{\eta}{\curvearrowright} Z$ a semigroup action by local homeomorphisms. For every $Y \subset Z$ and $s \in S$, we have $\eta_s^{-1}(\overline{Y}) = \overline{\eta_s^{-1}(Y)}$.
\end{lemma}
\begin{proof}
The map $\eta_s$ is continuous, so $\eta_s^{-1}(\overline{Y})$ is a closed subset of $Z$ containing $\eta_s^{-1}(Y)$ and hence $\overline{\eta_s^{-1}(Y)} \subset \eta_s^{-1}(\overline{Y})$. To prove the reverse inclusion, let $x \in \eta_s^{-1}(\overline{Y})$. Since $\eta_s$ is a local homeomorphism, there is an open neighbourhood $U$ of $x$ such that $\eta_s|_{U}: U \longrightarrow \eta_s(U)$ is a homeomorphism. Due to $\eta_s(x) \in \overline{Y}$, there is a net $(y_{\lambda})_{\lambda \in \Lambda} \subset Y$ such that $y_{\lambda} \stackrel{\lambda \rightarrow \infty}{\longrightarrow} \eta_s(x)$. Note that $\eta_s(U)$ is open and contains $\eta_s(x)$. Hence, we can assume $(y_{\lambda})_{\lambda \in \Lambda} \subset Y \cap \eta_s(U)$ without loss of generality. Now, $x_{\lambda} := \eta_s|_{U}^{-1}(y_{\lambda})$ defines a net $(x_{\lambda})_{\lambda \in \Lambda} \subset \eta_s^{-1}(Y) \cap U$ and continuity of $\eta_s|_{U}^{-1}$ gives $x_{\lambda} \longrightarrow x$. Therefore, we have shown that $x \in \overline{\eta_s^{-1}(Y)}.$     
\end{proof}  

\begin{corollary}\label{cor:orbit closure inv} 
Let $Z$ be a topological space, $S$ a commutative semigroup and $S \stackrel{\eta}{\curvearrowright} Z$ a semigroup action by local homeomorphisms. If $Y \subset Z$ is invariant, then so is $\overline{Y}$. In particular, the closure of the orbit $\CO(x)$ is invariant for every $x \in Z$.
\end{corollary}
\begin{proof}
For every $s \in S$, we get $\eta_s^{-1}(\overline{Y}) = \overline{\eta_s^{-1}(Y)} = \overline{Y}$ from Lemma~\ref{lem:preimage and closure} and the invariance of $Y$.  
\end{proof}

\begin{definition}\label{def:minimality for top sgp systems}
Let $Z$ be a topological space, $S$ a commutative semigroup and $S \stackrel{\eta}{\curvearrowright} Z$ a semigroup action by surjective local homeomorphisms. The dynamical system $(Z,S,\eta)$ is said to be \textit{minimal}, if $\emptyset$ and $Z$ are the only open invariant subsets of $Z$.
\end{definition}

\begin{remark}
In the above definition, one can replace open by closed. In \cite{EV}, this property is called irreducibility, possibly to avoid confusion with a notion of minimality apparently used for the groupoid picture.
\end{remark}

\begin{corollary}\label{cor:min via orbits}
A dynamical system $(Z,S,\eta)$ as in Definition~\ref{def:minimality for top sgp systems} is minimal if and only if $\CO(x) \subset Z$ is dense for all $x \in Z$.
\end{corollary}
\begin{proof}
This follows immediately from Corollary~\ref{cor:orbit closure inv}. 
\end{proof}

\noindent From now on, let us assume that $(X,P,\theta)$ is an irreversible $*$-commutative dynamical systems of finite type. The next proposition is based on \cite{EV}*{Proposition 11.1}.

\begin{proposition}\label{prop:min -> top free}
If $(X,P,\theta)$ is minimal, then it is topologically free.
\end{proposition}
\begin{proof}
Let us assume that $(X,P,\theta)$ is minimal, but not topologically free and derive a contradiction. Assume that there exist $p,q \in P$ with $p \neq q$ such that $\theta_{p|U} = \theta_{q|U}$ on some non-empty, open subset $U$ of $X$. Clearly, $\bigcup_{s,t \in P}{\theta_{s}^{-1}(\theta_{t}(U))} \subset X$ is invariant, non-empty and open. Since the dynamical system is minimal, this set is all of $X$. Since each $\theta_{s}^{-1}(\theta_{t}(U))$ is open and $X$ is compact, we can shrink the open cover $(\theta_{s}^{-1}(\theta_{t}(U)))_{s,t \in P}$ to a finite, open cover of $X$ given by $s_{1},\dots,s_n,t_1,\dots,t_n$. Next, fix an arbitrary $x \in X$ and let $i$ satisfy $x \in \theta_{s_i}^{-1}(\theta_{t_i}(U))$, i.e. there is $y \in U$ such that $\theta_{s_i}(x) = \theta_{t_i}(y)$. Then
\[\theta_{ps_i}(x) = \theta_{pt_i}(y) = \theta_{t_ip}(y) \stackrel{y \in U}{=} \theta_{t_iq}(y) = \theta_{qt_i}(y) = \theta_{qs_i}(x)\]   
and if we take $s := \bigvee\limits_{j = 1}^{n}{s_j}$, we get 
\[\theta_{ps}(x) = \theta_{s_i^{-1}s}\theta_{ps_i}(x) = \theta_{s_i^{-1}s}\theta_{qs_i}(x) = \theta_{qs}(x)\] 
for all $x$ in $X$. Hence, we have $\theta_{ps} = \theta_{qs}$. As $\theta_{(p \wedge q)s}$ is surjective and $P$ is commutative, this implies $\theta_{(p \wedge q)^{-1}p} = \theta_{(p \wedge q)^{-1}q}$. Without loss of generality, we can assume $(p \wedge q)^{-1}p \neq 1_P$, since $p \neq q$ forces $(p \wedge q)^{-1}p \neq 1_P$ or $(p \wedge q)^{-1}q \neq 1_P$. Using $\ast$-commutativity for $\theta_{(p \wedge q)^{-1}p}$ and $\theta_{(p \wedge q)^{-1}q}$ in the form of Proposition~\ref{prop: eq star ind}~(iii) yields 
\[\theta_{(p \wedge q)^{-1}p}^{-1}(\theta_{(p \wedge q)^{-1}q}(x)) = \theta_{(p \wedge q)^{-1}q}(\theta_{(p \wedge q)^{-1}p}^{-1}(x)) = \{x\}.\]
However, $(p \wedge q)^{-1}p \neq 1_P$ implies that the cardinality of the set on the left hand side is strictly larger than one, see Definition~\ref{def:SID}~(C). Thus, we obtain a contradiction.   
\end{proof}

\begin{theorem}\label{thm:simple iff minimal}
Let $(X,P,\theta)$ be an irreversible $*$-commutative dynamical system of finite type. Then the C*-algebra $\CO[X,P,\theta]$ is simple if and only if $(X,P,\theta)$ is minimal.
\end{theorem}
\begin{proof}
If we assume $\CO[X,P,\theta]$ to be simple, then $C(X)$ intersects every non-zero ideal in $\CO[X,P,\theta]$ non-trivially, so $(X,P,\theta)$ is topologically free by Theorem~\ref{thm:top free char}. Now suppose $\emptyset \neq U \subset X$ is invariant and open. Then 
\[\begin{array}{rclclcl}
\supp \alpha_{p}(f) &=& \theta_{p}^{-1}(\supp f) &\subset& \theta_{p}^{-1}(U) &=& U \\
\supp L_{p}(f) &\subset& \theta_{p}(\supp f) &\subset& \theta_{p}(U) &=& U 
\end{array}\]
holds for every $p \in P$ and $f \in C_{0}(U)$ because $U$ is invariant. From this we infer that the ideal $I$ in $\CO[X,P,\theta]$ generated by $C_{0}(U)$ satisfies $I \cap C(X) \subset C_{0}(U)$. But as $\CO[X,P,\theta]$ is simple and $U \neq \emptyset$, we have $I = \CO[X,P,\theta]$ and hence $U = X$.

Conversely, if $(X,P,\theta)$ is minimal and $0 \neq I$ is an ideal in $\CO[X,P,\theta]$, we have $I \cap C(X) = C_{0}(U)$ for some open $U \subset X$. Due to Proposition~\ref{prop:min -> top free}, $(X,P,\theta)$ is topologically free. Hence $U$ is non-empty according to Proposition~\ref{prop:top free gives IIP}. We claim that $U$ is invariant. To see why, let $x \sim y \in U$, i.e. there exist $p,q \in P$ such that $\theta_p(x) = \theta_q(y)$. Pick a non-negative function $f \in C_{0}(U)$ satisfying $f(y) > 0$ (such an $f$ always exists as $U$ is open and $X$ is a normal space). Additionally, choose $(\nu_i)_{1 \leq i \leq n}$ for $\theta_p$ as in Lemma~\ref{lem:loc homeo right reconstruction}. Using the relations (I),(II) and (IV) for $\CO[X,P,\theta]$, we get  
\[\begin{array}{c} \alpha_p(L_q(f)) = \sum\limits_{1 \leq i \leq n} \nu_is_ps_q^*fs_qs_p^*\nu_i \in I, \end{array}\]
which shows $\alpha_p(L_q(f)) \in C_{0}(U)$. Moreover, we have
\[\alpha_p(L_q(f))(x) = L_q(f)(\theta_p(x)) \geq N_q^{-1}f(y) > 0\]
because $f$ is non-negative and $y \in \theta_q^{-1}(\theta_p(x))$. Thus $x \in \supp \alpha_p(L_q(f)) \subset U$, so $U$ is invariant by Lemma~\ref{lem:inv preimage}. Since $U$ is a non-empty, invariant open subset of $X$, minimality forces $U = X$ and hence $I = \CO[X,P,\theta]$. Hence $\CO[X,P,\theta]$ is simple.
\end{proof}

\noindent Coming back to irreversible algebraic dynamical systems, we recall that we can only treat commutative irreversible algebraic dynamical systems of finite type within the framework of irreversible $*$-commutative dynamical systems of finite type, see Corollary~\ref{cor:CIAD vs SIDoFT}:

\begin{corollary}\label{cor:char min-simple for CIADoFT}
A commutative irreversible algebraic dynamical system of finite type $(G,P,\theta)$ is minimal if and only if the C*-algebra $\CO[G,P,\theta]$ is simple.
\end{corollary}
\begin{proof}
By Corollary~\ref{cor:CIAD vs SIDoFT}, $(\hat{G},P,\hat{\theta})$ is an irreversible $*$-commutative dynamical system of finite type. According to Proposition~\ref{prop:consistent constr for CIADoFT}, $\CO[G,P,\theta]$ is isomorphic to $\CO[\hat{G},P,\hat{\theta}]$. It is easy to see that $(G,P,\theta)$ is minimal precisely if the union of the kernels $(\ker\theta_p)_{p \in P}$ is dense in $\hat{G}$. In other words, the orbit of $1_{\hat{G}}$ in the sense of Definition~\ref{def:orbit sgp action} is dense in $\hat{G}$. Since $\hat{G}$ is a group and we are dealing with group endomorphisms, $\CO(\chi)$ is dense in $\hat{G}$ for each $\chi \in \hat{G}$. By Corollary~\ref{cor:min via orbits}, this is equivalent to minimality of the topological dynamical system $(\hat{G},P,\hat{\theta})$. Now the claim follows directly from Theorem~\ref{thm:simple iff minimal}.
\end{proof}

\end{document}